\documentclass[a4paper, 11pt,reqno]{amsart}
\usepackage{amsmath,amsthm,amsbsy,amssymb, mathtools}

\makeatletter
\def\l@section{\@tocline{1}{0pt}{1pc}{}{}}
\def\l@subsection{\@tocline{2}{0pt}{1pc}{4.6em}{}}
\def\l@subsubsection{\@tocline{3}{0pt}{1pc}{7.6em}{}}
\renewcommand{\tocsection}[3]{%
  \indentlabel{\@ifnotempty{#2}{\makebox[2.3em][l]{%
    \ignorespaces#1 #2.\hfill}}}#3}
\renewcommand{\tocsubsection}[3]{%
  \indentlabel{\@ifnotempty{#2}{\hspace*{2.3em}\makebox[2.3em][l]{%
    \ignorespaces#1 #2.\hfill}}}#3}
\renewcommand{\tocsubsubsection}[3]{%
  \indentlabel{\@ifnotempty{#2}{\hspace*{4.6em}\makebox[3em][l]{%
    \ignorespaces#1 #2.\hfill}}}#3}
\makeatother

\usepackage{arydshln} 
\usepackage{booktabs} 

\usepackage{array}   

\makeatletter
\newcommand{\leqnos}{\tagsleft@true\let\veqno\@@leqno}
\newcommand{\reqnos}{\tagsleft@false\let\veqno\@@eqno}
\reqnos
\makeatother

\usepackage{enumitem} 
\makeatletter
\def\namedlabel#1#2{\begingroup
    #2%
    \def\@currentlabel{#2}%
    \phantomsection\label{#1}\endgroup
}
\makeatother

\usepackage{xcolor} 
\definecolor{orange}{rgb}{1,0.5,0}
\definecolor{Red}{rgb}{.795,0.015,0.017}
\definecolor{Ggreen}{rgb}{0.,0.675,0.0128}
\definecolor{Bblue}{rgb}{0.16,.32,0.91}
 \usepackage[backref=page]{hyperref}
\hypersetup{
    colorlinks = true,
linkcolor={red},
urlcolor={blue},
citecolor={Ggreen},    
urlcolor = {blue},
citebordercolor = {0.33 .58 0.33},
 linkbordercolor = {0.99 .28 0.23},
 pdftitle={Statistical Distribution of Fermat Quotients},
 pdfauthor={ACVZ}
}
\usepackage{mathabx}

\usepackage{mathrsfs}
\newcommand{\scr}[1]{\mathscr #1}

\newcommand{\scrG}{\mathscr G}

\def\cB{\mathcal B}

\def\cD{\mathcal D}

\newcommand{\cH}{\mathcal H}

\def\cM{\mathcal M}

\def\cS{\mathcal S}

\newcommand{\edv}{\mathrel\Vert} 
\renewcommand{\edv}{\mathbin\Vert} 
\renewcommand{\edv}{\mathbin\|} 

\usepackage{bm}
\newcommand*{\B}[1]{\ifmmode\bm{#1}\else\textbf{#1}\fi}

\newcommand{\be}{\protect\B{e}}

\newcommand{\bj}{\protect\B{j}}
\newcommand{\br}{\protect\B{r}}

\newcommand{\bw}{\protect\B{w}}
\newcommand{\bx}{\protect\B{x}}
\newcommand{\by}{\protect\B{y}}

\newcommand{\bzero}{\mathbf{0}}

\def\FF{\mathbb{F}}

\def\ZZ{\mathbb{Z}}
\def\QQ{\mathbb{Q}}

\newcommand{\sdfrac}[2]{\mbox{\small$\displaystyle\frac{#1}{#2}$}}

\usepackage{xspace}

\DeclareMathOperator{\K}{K} 

\DeclareMathOperator{\Av}{A} 
\DeclareMathOperator*{\area}{area}

\DeclareMathOperator*{\Id}{Id}
\renewcommand{\Id}{Id} 

\DeclareMathOperator*{\id}{id}

\DeclareMathOperator*{\length}{length}
\DeclareMathOperator{\diam}{diam}

\makeatletter
\renewcommand\xleftrightarrow[2][]{\ext@arrow 0099{\longleftrightarrowfill@}{#1}{#2}}
\def\longleftrightarrowfill@{\arrowfill@\leftarrow\relbar\rightarrow}
\makeatother

\usepackage{float}

\renewcommand{\pmod}[1]{\left( \mathrm{ mod\;}#1\right)}

\newcommand{\abs}[1]{\left| #1 \right|}

\theoremstyle{plain}
\newtheorem{theorem}{Theorem}

\newtheorem{lemma}{Lemma}[section]
\newtheorem{proposition}{Proposition}[section]

\theoremstyle{remark}
\newtheorem{remark}{Remark}[section]
\newtheorem*{remark*}{Remark}

\theoremstyle{definition}

\usepackage{subfigure}  
\usepackage{float}
\usepackage[pdftex]{graphicx}
\graphicspath{{IMAGES-ORBITS/}{IMAGES/}{./}}
\usepackage[]{caption}
\renewcommand*{\backref}[1]{}
\renewcommand*{\backrefalt}[4]{%
  \ifcase #1 %
No citations.
  \or
(page #2).%
  \else
(pages #2).%
  \fi%
}
\usepackage{cite}

\usepackage{pgfplots}
\pgfplotsset{compat=1.8}
\usetikzlibrary{spy}  

\usepackage{nicematrix}
\makeatletter
\renewcommand*\env@matrix[1][\arraystretch]{%
  \edef\arraystretch{#1}%
  \hskip -\arraycolsep
  \let\@ifnextchar\new@ifnextchar
  \array{*\c@MaxMatrixCols c}}
\makeatother

\usepackage{stmaryrd}
\newcommand{\List}[1]{\llbracket #1 \rrbracket}

\usepackage[]{mdframed}

\usepackage{tikz-cd} 

\begin{document}
\title[Twisted aughts of alternating involutions]
{Twisted aughts of alternating involutions}

\author[R. N. Bhat]{Raghavendra N. Bhat}
\address[Raghavendra N. Bhat]{Department of Mathematics, University of Illinois at Urbana-Champaign, 1409 West Green 
Street, Urbana, IL 61801, USA}
\email{rnbhat2@illinois.edu}

\author[C. Cobeli]{Cristian Cobeli}
\address[Cristian Cobeli]{``Simion Stoilow'' Institute of Mathematics of the Romanian Academy,~21 Calea Griviței Street, P. O. Box 1-764, Bucharest 014700, Romania}
\email{cristian.cobeli@imar.ro}

\author[S. Iwai]{Shuta Iwai}
\address[Shuta Iwai]{Department of Mathematics, University of Illinois at Urbana-Champaign, 1409 West Green 
Street, Urbana, IL 61801, USA}
\email{siwai2@illinois.edu}

\author[Z. Ye]{Zimeng Ye}
\address[Zimeng Ye]{Department of Mathematics, University of Illinois at Urbana-Champaign, 1409 West Green 
Street, Urbana, IL 61801, USA}
\email{zimengy3@illinois.edu}

\author[A. Zaharescu]{Alexandru Zaharescu}
\address[Alexandru Zaharescu]{Department of Mathematics, University of Illinois at Urbana-Champaign, 1409 West Green 
Street, Urbana, IL 61801, USA,
%
and 
``Simion Stoilow'' Institute of Mathematics of the Romanian Academy,~21 
Calea Griviței 
Street, P. O. Box 1-764, Bucharest 014700, Romania}
\email{zaharesc@illinois.edu}



\begin{abstract}
Let $\cM(n)$ be the subgroup of $GL(n,\ZZ)$ generated 
by the particular involutions 
that are identical to the identity, except for a single line where 
$-1$ and $+1$ alternate.
 We study the properties of $\cM(n)$, and then find several notable 
characteristics of the unions of trajectories obtained by iteratively applying a 
fixed sequence of such involutions to elements from $\ZZ^n$.
\end{abstract}
\maketitle
{
  \hypersetup{linkcolor=blue}
  \tableofcontents
}

\section{Introduction}
\bigskip

In the quest for distinguished patterns in tiling problems or partitioning the space with integers, certain involutions arise and prove to have a directive role.
Depending on their destined play, they can function in a lower dimension, as in~\cite{BCZ2024}, where they were used to indicate location and content 
by changes in direction and size, or by keeping the dimension as in~\cite{CZ2024, CRZ2025}, where they also take on the property of being invariant under translations.

Within the same class, the type of operators we are considering are elements of $GL(n,\ZZ)$,
which act on elements of $\ZZ^n$ 
by replacing one component with an alternating sum of all components to which an additive constant is added while the other components remain unchanged.
Nonetheless, that constant added to the alternating sum
entangles a straightforward generalization of these transformations
to higher dimensions.
In this case, by selecting a feasible option, we have opted to put the constant aside, 
along with the requirement of invariance under translations, 
keeping only the one of being involutions.
Moreover, in the most basic considered version, we can replace $\ZZ$ with $\FF_2$,
but to emphasize the reflective property, we will keep the entries in $\{-1,0,1\}$ instead.

Let $n$ be a positive integer. 
For any integer 
$j$, with $1\le j\le n$,
we let $K_j = K_j(n)$ denote the  modified identity matrix such that the $j$-th line 
starts with $(-1)^j$ and then alternates~$\pm 1$'s.
To be precise, let 
$\be_j=\be_j(n)$ be the $j$-th standard basis 
(column) vector
and let
$\br_j=\br_j(n):=\big((-1)^j,(-1)^{j+1},\dots, (-1)^{j+n-1}\big)$.
(We will omit the reference to dimension $n$ as the dimension will be clear from the context.)
Thus
\begin{align}\label{eqDefKj}
  K_j:= \Id - \be_j \be_j^T 
  + \be_j\br_j.
\end{align}
We also consider the corresponding operators
$\K_j : \ZZ^n \to\ZZ^n$ that transform any vector 
$\bx=(x_1,\dots,x_n)\in\ZZ^n$ into the product $(K_j\bx^T)^T$.
(Further on, in this context, we will drop the indication 
of transpositions by writing simply $\K_j(\bx)$ or just $K_j\bx$.)
Iteratively applying operators $\K_j$, with indices selected from 
a given sequence of positive integers, 
creates a subset of $\ZZ^n$, called \textit{trajectory} or \textit{orbit}, when it is closed. 
These trajectories vary based on the chosen initial point, called 
trajectory \textit{generator}.
In the case of non-trivial sequences,
excluding generators belonging to certain special 
locations that are part of sets with zero measure in the 
working environment (such as the origin, specific axes, 
planes, or hyperplanes), the trajectories share the same 
class of shape indicated by the same underlying topology.
For any fixed sequence of indices of the operators, 
the union of the associated trajectories forms a set 
partition of $\ZZ^n$.
The mentioned additive constant, which is~$0$ in the case 
of operators~$\K_j$, causes the trajectories generated by 
their iterated application to be bounded. 
This is in contrast to, for instance, the trajectories 
that were unbounded parallel parabolas
or intricately looking circles that partition the plane 
$\ZZ^2$ in~\cite{CZ2024, CRZ2025}.

\begin{figure}[thb]
\centering
\hfill
\includegraphics[width=0.315\textwidth,angle=-90]{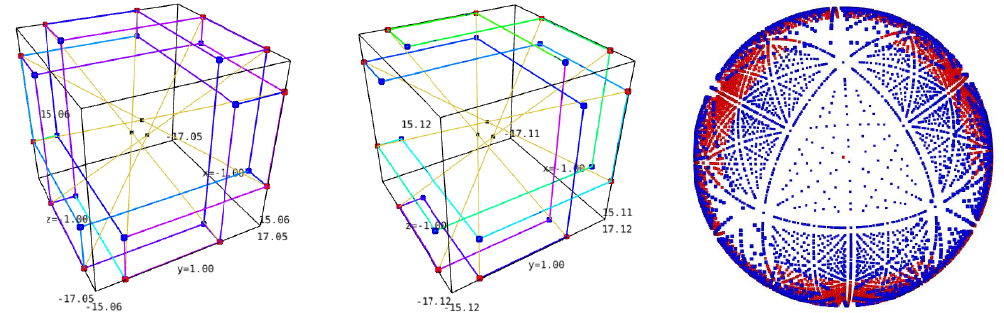}
\hfill\mbox{}
\caption{
Let $\bw = (10,8,15)$ and
$\K = 
\left(\K_3(\K_2\K_1)^2\K_2
\K_3(\K_1\K_2)^2\K_1
\right)^2
$.
The image on the left-side shows 
the $24$ nodes and $36$ edges of the graph obtained by starting from~$\bw$ and applying recursively any of the operators $\K_1,\K_2,\K_3$. The Euclidean diameters of the graph and their midpoints are also included. 
The Eulerian-Hamiltonian cycle, which connects the nodes of the orbit $o(\bw)$ in the order given by their appearance while applying the operators in $\K$, is shown in the middle image.
In the right-hand image, one can see the projections onto the unit sphere of all the nodes of the orbits generated likewise by~$\K$ while starting from every $\bx \in [0,10]^3$.
(In the electronic version, the diametrically opposite nodes of the orbits they belong to are in red and the others are in blue.)
}
\label{Figure3dorbits}
\end{figure}

Applying the $K_j$ operators iteratively on points from 
$\ZZ^n$, the ubiquitous geometric shape that arises in all 
pairs of dimensions is that of a \textit{twisted aught}, resembling that of tilted figure~$8$, with lines drawn parallel to the coordinate axes.
(See their algebraic expressions and graphic representations in Sections~\ref{SectionAughts} and~\ref{SectionAughtsGraphics}.)
In Figure~\ref{Figure3dorbits}, one can see the graph with $24$ nodes and $36$ 
edges that contains any possible trajectory starting from $(10, 8, 15)$.
The nodes are positioned in $12$ levels, with $4$ 
parallel to each of the three planes 
$x_j=0$ for $j = 1,2,3$, respectively.
In the orbit shown in the center of Figure~\ref{Figure3dorbits}, 
even though there are $6$ nodes for each aught on each level,
some edges are left untraced when 
jumping from one level to another, thus completing a Hamiltonian path, which is also an Eulerian cycle.

There are orbits that also have notable arithmetic properties
such as those related to squares and primes.
Thus, an orbit that
passes through a point with coordinates $(a^2,c^2)$, 
where $a^2+b^2=c^2$ for some integers $a,b,c$,
contains only nodes whose 
coordinates are squares or their negatives;
and an orbit that passes through $(p,p+2)$, where both $p$ and~$p+2$ are 
primes, has only prime numbers or their opposites in all components of the 
other nodes.


A wide range of topics regarding problems related to groups generated by involutions and matters on the associated Hamiltonian cycles have been studied in many works;
see for example~\cite{Lee2021, SC1994, BV2010, RST2017, KM2008, CSL2015, GY1996}, the surveys~\cite{PR2009, LPRTW2019, KM2009}, and the references therein.

\subsection{Main results}
Our work has two main parts. 
In Sections~\ref{SectionLemmas}--\ref{SectionIsomorphismMS}, starting from the observation that 
all $K_j$ with $j=1,2,\dots,n$ are involutions, 
we follow step by step on their properties and
arrive at a comprehensive description of the group they 
generate (see Theorem~\ref{LemmaMsihMultiplication}). 
Lately, we find that, as a Coxeter group, 
$\langle K_1,K_2,\dots,K_n\rangle$ is isomorphic 
with the symmetric group~$S_{n+1}$ (see Section~\ref{SectionIsomorphismMS}).
In the second part, we turn our attention specifically to the characteristic 
aspects of sets formed by unions of orbits, particularly in the two-dimensional 
case.
This proves to be essential for understanding the general phenomenon, 
which will be the topic of a subsequent work.

Starting from an arbitrary $\bx\in\ZZ^2$ and applying the $K_j$ operators 
iteratively, we obtain a cyclic trajectory $o(\bx)$ 
with $6$ points, 
except for a few degenerate situations. 
In each orbit $o(\bx)$, let's color in red the points farthest from each other, 
that is, those that are diametrically opposed in Euclidean distance.
We color the points that are not diametrically opposed in blue.
By letting $\bx$ run in increasingly larger sets, we obtain different shapes 
that expand and eventually lead to a complete coloring of $\ZZ^2$ (see Figure~\ref{FigureEuclideanDiametersProjections}).
The following theorem shows the common property of all orbits to be 
specifically colored based on the positioning of their nodes in the plane.

\begin{theorem}\label{TheoremProjection2D}
  In the Euclidean distance, the diametrically opposite nodes of 
the bounding boxes enclosing the two-dimensional  
orbits project onto the unit circle in two opposite arcs of length 
\mbox{$\arctan\left(2\right) - \arctan\left(1/2\right)$}
each.
Then, $(\arctan\left(2\right) - \arctan\left(1/2\right))/\pi 
\approx 0.204833$ is,
in this sense, 
the limit probability that a lattice point belongs to the Euclidean diameter of the orbit to which it belongs.
\end{theorem}

\begin{figure}[tb]
\centering
\hfill
\includegraphics[width=0.32\textwidth]{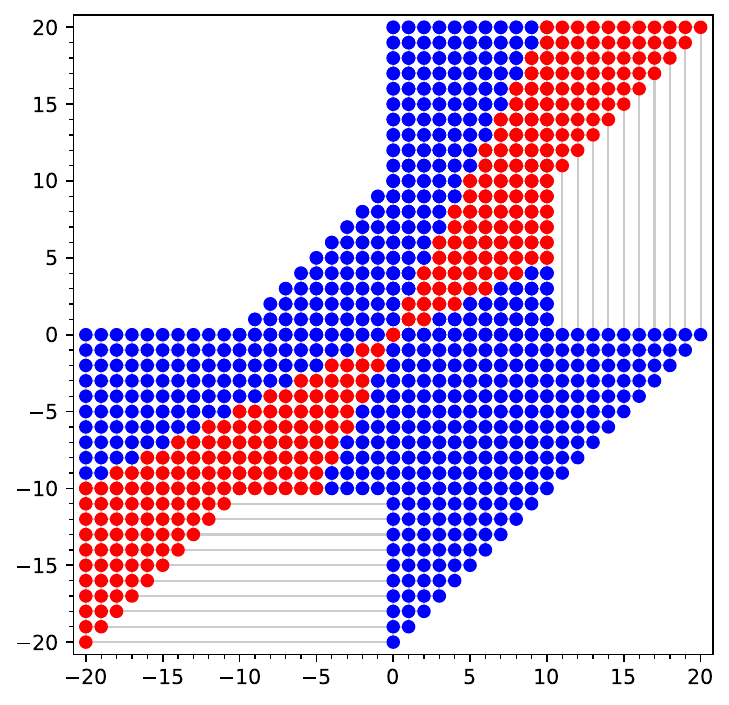}
\includegraphics[width=0.32\textwidth]{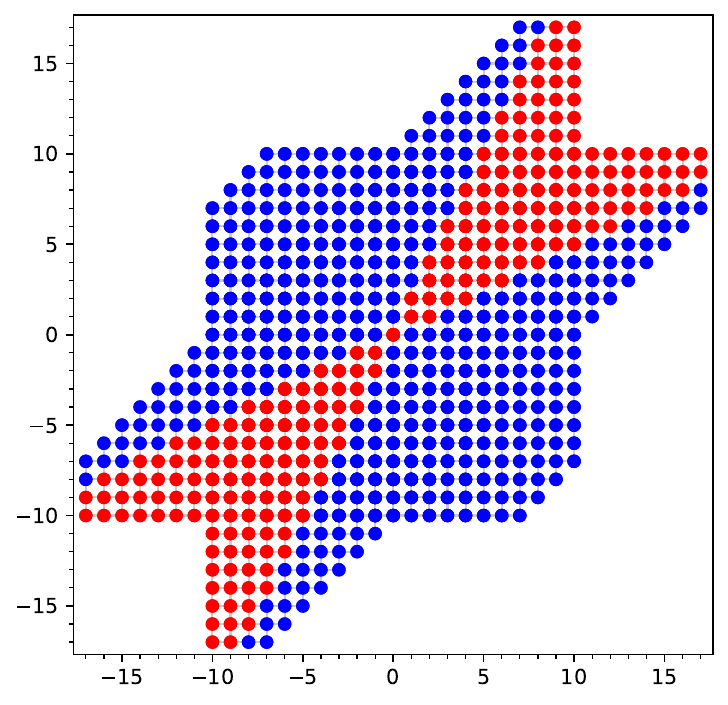}
\includegraphics[width=0.32\textwidth]{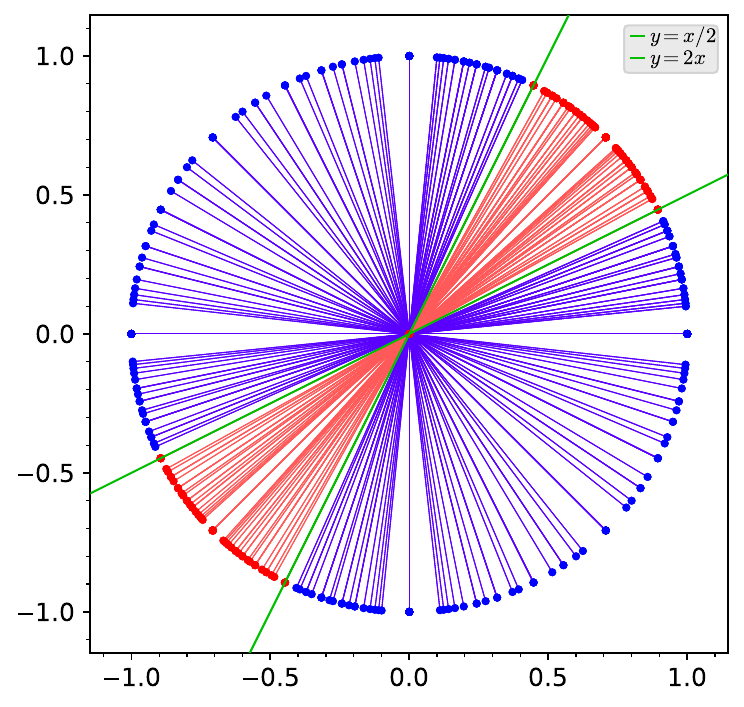}
\hfill\mbox{}
\caption{The orbits generated by $\bx\in[0,M]\times[0,2M]$ (left)
and the orbits generated by $\bx\in[-M,M]^2$ with $d_E(O,\bx)\le M$ (middle).
Points are colored in red if they are part of the Euclidean diameter of 
the orbit they belong to and in blue otherwise. 
The same colors are used to draw the projections of the orbits generated by
$\bx\in [0,M]^2$ onto the unit circle (right).
(In all three representations, $M=10$.)
}
\label{FigureEuclideanDiametersProjections}
\end{figure}

We will see in Section~\ref{SectionSqrt5} 
that the length of the orbits generated by points 
$\bx\in\ZZ^2$ at a Euclidean distance of at most~$R$ from the origin 
ranges between $0$ and $4\sqrt{5}R\approx 8.94427R$, 
the latter being attained for lattice points close to
$(R\cos(\phi), -R\sin(\phi))$, 
where $\phi=\arctan 2$.
The next theorem proves that the average of these perimeters is
about $4.99649R$.
\begin{theorem}\label{TheoremPerimeterOrbitsDisk}
If $R$ is a positive integer, then the average length of an orbit
$o(\bx)$, where $\bx$ runs over all lattice points with $\Vert\bx\Vert_E\le R$ is 
$ \frac{8}{3\pi}\big(\sqrt{2} + 2\sqrt{5}\big)R +O(1)$.
\end{theorem}

The modular representation of orbit perimeters 
creates various patterns 
depending on the modulus $d$.
Some of these are included in Figure~\ref{FigureColoredOrbits}, and 
their number is estimated in the following theorem.
\begin{theorem}\label{Theorem_rmodd}
Assume $d\ge 2$, 
$r$ with $0\le r\le d-1$, and $M\ge 1$ are integers.
Let $N_{r,d}(M)$ be the tally of the orbits $o(\bx)$ with $\bx\in[0,M]^2\cap\ZZ^2$
and length congruent to 
$r\pmod d$.
Then 
\begin{align*}
  N_{r,d}(M) = 
  \begin{cases}
     0, & \text{ if } 2\edv d \text{ and } 2\not\mid r, \text{ or } 4\mid d \text{ and } 4\not\mid r;\\[2pt]
     \frac{1}{d}M^2 + O(M), & \text{ if } 2\edv d \text{ and } 2\mid r;\\[2pt]
     \frac{2}{d}M^2 + O(M), & \text{ if }  4 \mid d \text{ and } 4\mid r;\\[2pt]
     \frac{1}{2d}M^2 + O(M), & \text{ if }  2\not\mid d.
  \end{cases}
\end{align*}
\end{theorem}

\begin{figure}[htb]
\centering
\hfill
\includegraphics[width=0.32\textwidth]{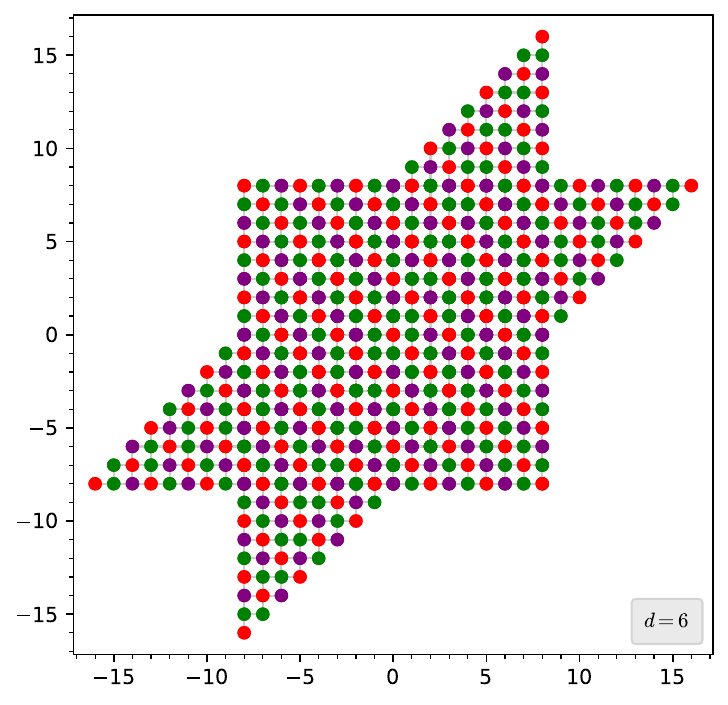}
\includegraphics[width=0.32\textwidth]{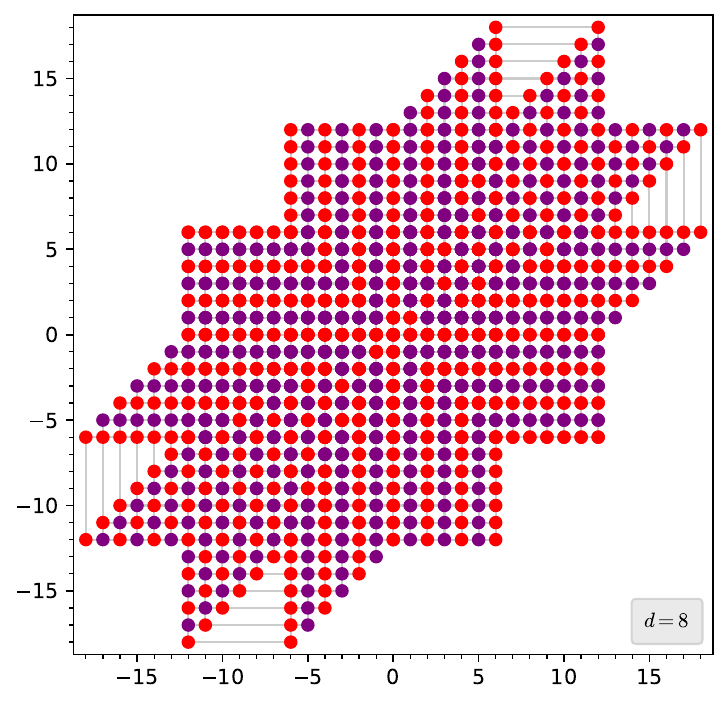}
\includegraphics[width=0.32\textwidth]{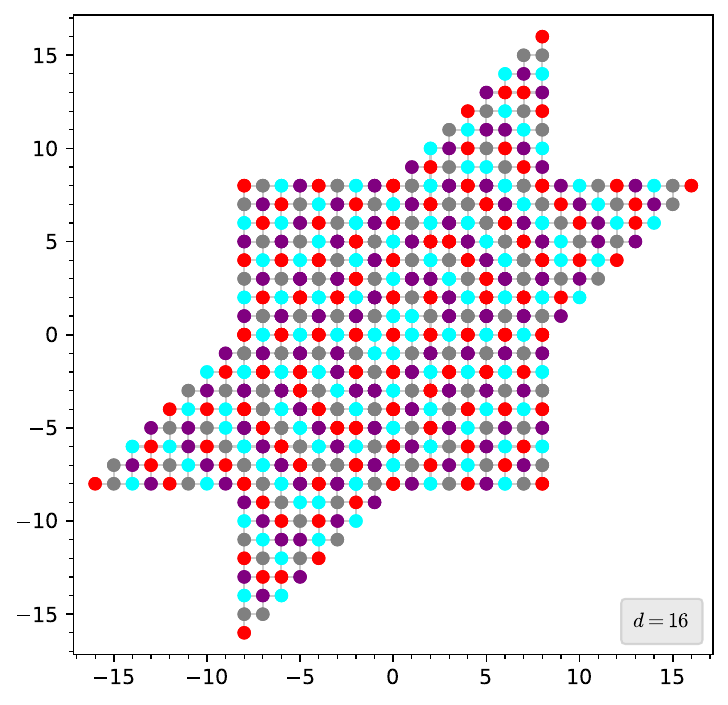}
\hfill\mbox{}\\
\hfill
\includegraphics[width=0.32\textwidth]{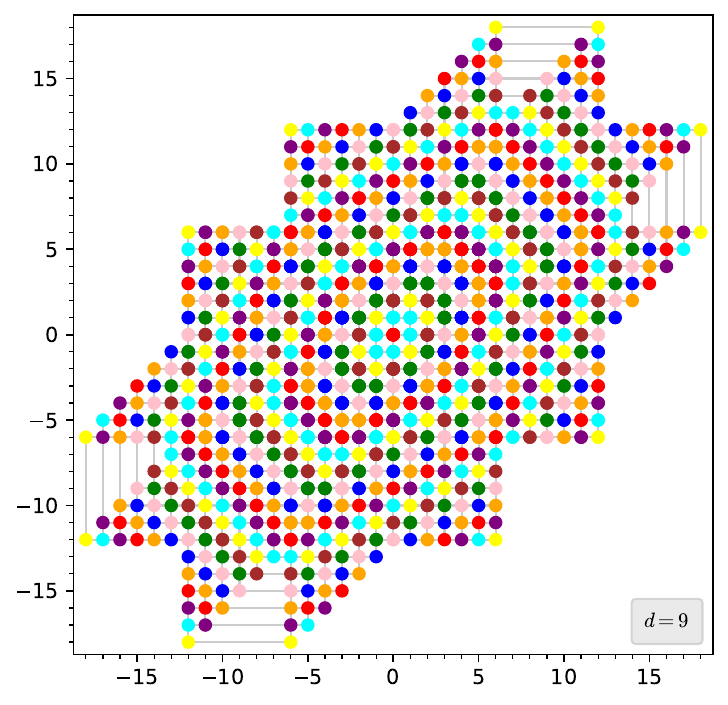}
\includegraphics[width=0.32\textwidth]{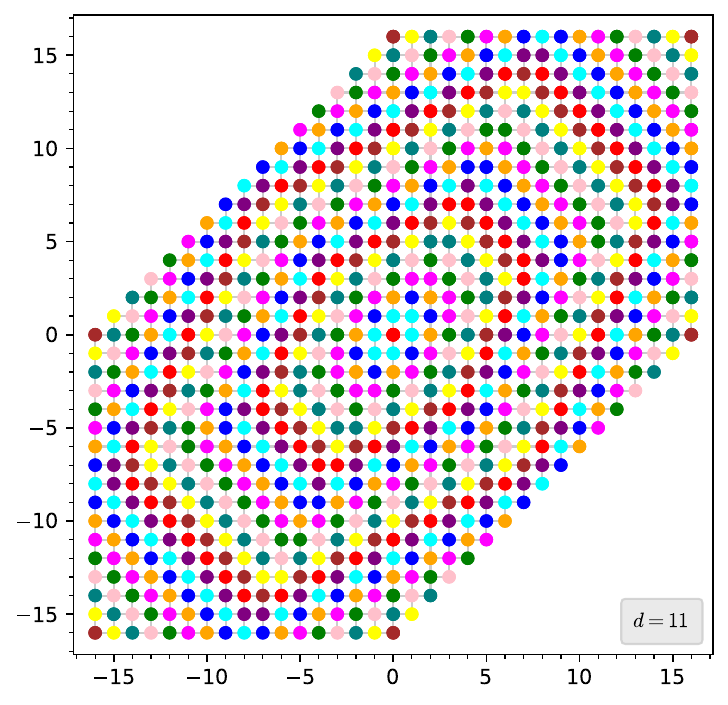}
\includegraphics[width=0.32\textwidth]{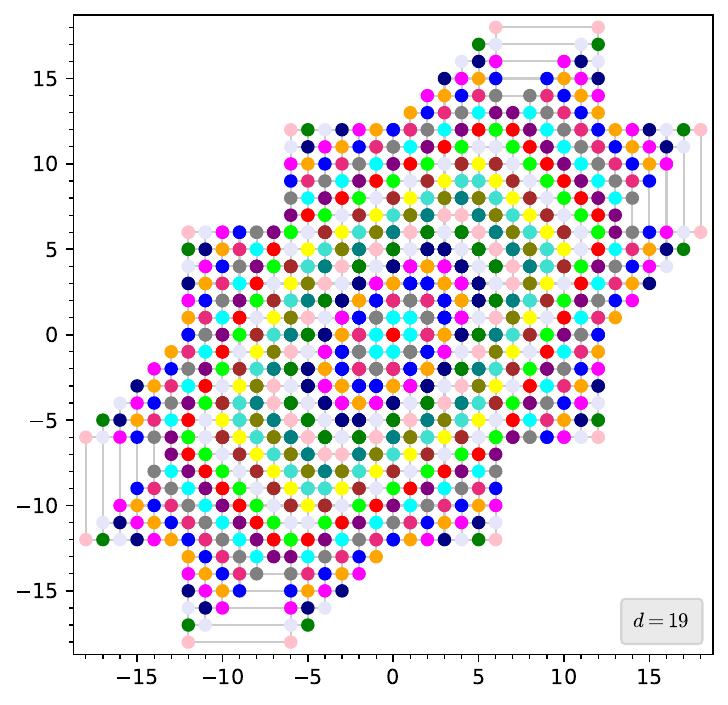}
\hfill\mbox{}
\caption{Colored representation of the orbits by their perimeter mod $d$.
The images are grouped by parity, with $d = 6, 8, 16$ at the top, 
and $d = 9, 11, 19$ at the bottom.
The generators $\bx$ of the orbits $o(\bx)$ shown in the first and third images are in $[-8,8]^2$, 
the generators in the fifth are in~$[0,16]^2$, and the generators in the other `butterfly shape' are in $[-12,12]\times [0,6]$.
}
\label{FigureColoredOrbits}
\end{figure}
Additional topics such as
finding an exact formula for the number of orbits with a given perimeter 
or 
different approaches in estimating possible means of diameters or 
perimeters of orbits are further discussed in Section~\ref{SectionArithmeticOn2Dorbits}.

\section{\texorpdfstring{The involutions $K_j$}{The involutions Kj}}\label{SectionLemmas}

Let us begin by noticing  that the difference 
$\Id - \be_j \be_j^T $ 
sets the $j$th line of the identity matrix to zero
and $\be_j\br_j$ is the modified-zero matrix with the $j$th line replaced by $\br_j$, 
so that the operators 
$K_j= \Id - \be_j \be_j^T + \be_j\br_j$ defined
by~\eqref{eqDefKj} 
indeed satisfy the intended setup.

For example, if $n=3$, the operators are:
\begin{align*}
    K_1 = 
\begin{pmatrix}
    -1 & 1 & -1\\
    0 & 1 & 0 \\
    0 & 0 & 1
\end{pmatrix},
\ 
    K_2 = 
\begin{pmatrix}
    1 & 0 & 0\\
    1 & -1 & 1 \\
    0 & 0 & 1
\end{pmatrix},
\text{ and }
    K_3 = 
\begin{pmatrix}
    1 & 0 & 0\\
    0 & 1 & 0 \\
    -1 & 1 & -1
\end{pmatrix}.
\end{align*}

Here, we list the main operations that drive the grammar in the group
$\langle K_1, K_2, \dots, K_n \rangle$.

\begin{remark}\label{Remark_G}
   Let $n,j$, and $\ell$ be integers and suppose $1\le j,\ell\le n$.
 Then:
\begin{enumerate}[itemsep=4pt] 
  \item[\namedlabel{Remark_GA}{\normalfont{(\textup{1})}}]
The following products are $1\times 1$ matrices:
$\be_j^T\be_j=(1)$;
$\br_j\be_j=-(1)$;
$\be_j^T\br_j^T=-(1)$; 
and~$\br_j\br_j^T= n(1)$.
Also, we have:
$\br_j\be_\ell = \br_\ell\be_j = (-1)^{j+l-1}(1)$.
  \item[\namedlabel{Remark_GB}{\normalfont{(\textup{2})}}]
The $n\times n$ matrix $\be_j\be_\ell^T$ has $1$ at position $(j,\ell)$
and zeros elsewhere.
  \item[\namedlabel{Remark_GC}{\normalfont{(\textup{3})}}]
  The matrix $\be_j\br_j$ replaces the $j$-th line of the zero matrix 
  with $\br_j$.
  \item[\namedlabel{Remark_GD}{\normalfont{(\textup{4})}}]
 If $j\neq\ell$, then $\be_j^T\be_\ell=(0)$.
  \item[\namedlabel{Remark_GE}{\normalfont{(\textup{5})}}]
  If $j\neq\ell$, then
   $\be_\ell\br_\ell \be_j\br_j = -\be_\ell\br_\ell$.
\end{enumerate} 
\end{remark}

The next lemma verifies that $K_j$ is an involution.
\begin{lemma}\label{LemmaInvolutionK}
   Let $n$ and $j$ be integers, and suppose 
   $1\le j\le n$.
 Then  $K_j^2=\Id$.
\end{lemma}
\begin{proof}
The definition~\eqref{eqDefKj} implies
\begin{equation*}
  \begin{split}
   K_j^2 &= \big(\Id - \be_j \be_j^T + \be_j\br_j\big)
            \big(\Id - \be_j \be_j^T + \be_j\br_j\big)\\
   &= \Id- 2\be_j \be_j^T + 2\be_j\br_j
    + \be_j (\be_j^T\be_j) \be_j^T 
    - \be_j (\be_j^T\be_j)\br_j
   -  \be_j(\br_j\be_j) \be_j^T +  \be_j(\br_j\be_j)\br_j,
 \end{split}
\end{equation*}
where we have associated the matrices in the middle whose product is $1$-dimensional.
Then, observing by~Remark~\ref{Remark_G}.\ref{Remark_GA} that $\be_j^T\be_j = (1)$
and  $\br_j\be_j = -(1)$, it follows that
\begin{align*}
   K_j^2 &= \Id- 2\be_j \be_j^T + 2\be_j\br_j
    + \be_j  \be_j^T 
    - \be_j \br_j
   +  \be_j \be_j^T - \be_j\br_j
   =\Id.
\end{align*}
\end{proof}

Next we calculate the product of two distinct $K$-operators. 
\begin{lemma}\label{LemmaFormulaKjl}
   Let  $j,\ell$, and $n$ be integers and suppose $1\le j\neq \ell\le n$.
 Then 
\begin{align}\label{eqLemmaFormulaKjl}
    K_jK_\ell  
   = \Id- \be_j \be_j^T - \be_\ell \be_\ell^T
    + \be_\ell\br_\ell 
    + (-1)^{j+\ell}\be_j \be_\ell^T.
 \end{align}
\end{lemma}
\begin{proof}
By definition~\eqref{eqDefKj}, we have:
\begin{equation}\label{eqLemmaFormulaKjl2}
  \begin{split}
   K_jK_\ell &= \big(\Id - \be_j \be_j^T + \be_j\br_j\big)
            \big(\Id - \be_\ell \be_\ell^T + \be_\ell\br_\ell\big)\\
   &= \Id- \be_j \be_j^T - \be_\ell \be_\ell^T
   + \be_j\br_j + \be_\ell\br_\ell\\
  &   \quad 
    + \be_j (\be_j^T\be_\ell) \be_\ell^T 
    - \be_j (\be_j^T\be_\ell)\br_\ell
   -  \be_j(\br_j\be_\ell) \be_\ell^T +  \be_j(\br_j\be_\ell)\br_\ell.
 \end{split}
\end{equation}
According to~Remark~\ref{Remark_G}.\ref{Remark_GD} $j\neq\ell$ implies $\be_j^T\be_\ell=(0)$, so that
the first two terms on the last line of~\eqref{eqLemmaFormulaKjl2} are $\bzero$ (the $n\times n$ zero matrix).
Thus
\begin{equation}\label{eqLemmaFormulaKjl3}
  \begin{split}
   K_jK_\ell  
   &= \Id- \be_j \be_j^T - \be_\ell \be_\ell^T
   + \be_j\br_j + \be_\ell\br_\ell\\
  &   \quad 
   -\be_j(\br_j\be_\ell) \be_\ell^T +  \be_j(\br_j\be_\ell)\br_\ell.
 \end{split}
\end{equation}
By~Remark~\ref{Remark_G}.\ref{Remark_GA}, the last two terms above are:
\begin{align*}
   -\be_j(\br_j\be_\ell) \be_\ell^T +  \be_j(\br_j\be_\ell)\br_\ell  
   = (-1)^{j+\ell-1}(-\be_j \be_\ell^T+\be_j\br_\ell ).
\end{align*}
Replacing this expression in~\eqref{eqLemmaFormulaKjl3} and rearranging the terms, we obtain
\begin{align}\label{eqLemmaFormulaKjl5}
  K_jK_\ell  
   &= \Id- \be_j \be_j^T - \be_\ell \be_\ell^T
    + \be_\ell\br_\ell 
    + (-1)^{j+\ell}\be_j \be_\ell^T
   + \be_j\big(\br_j + (-1)^{j+l-1}\br_\ell\big).
\end{align}
By checking all possible parity cases, we find that the last term cancels out
yielding~\eqref{eqLemmaFormulaKjl} and concluding the proof.
\end{proof}
Thus, Lemma~\ref{LemmaFormulaKjl} shows that the product 
$K_jK_\ell$ 
product is the identity matrix transformed by replacing line $\ell$ with $\br_\ell$, 
and on line $j$ the element $(j,j)$ is replaced with $0$, 
and the element $(j,\ell)$ is replaced with $(-1)^{j+l}$.
For example, if $n=4$ we have:
\begin{align*}
  K_1K_4 = 
\begin{pmatrix}
    0 & 0 & 0 & -1 \\
    0 & 1 & 0 & 0 \\
    0 & 0 & 1 & 0 \\
    1 & -1 & 1 & -1
\end{pmatrix}
  \ \ \text{ and }\ \
  K_2K_3=
\begin{pmatrix}
    1 & 0 & 0 & 0 \\
    0 & 0 & -1 & 0 \\
    -1 & 1 & -1 & 1 \\
    0 & 0 & 0 & 1
\end{pmatrix}.
\end{align*}
Following the sequential transformations in the proof
of Lemma~\ref{LemmaFormulaKjl} 
step by step, we see that
in the first line of relation~\eqref{eqLemmaFormulaKjl3}, 
in the identity matrix the lines $j$ and $\ell$ 
are replaced with 
$\br_j$ and $\br_l$ respectively.
Then, after the cancellation of the last term in~\eqref{eqLemmaFormulaKjl5} the entire line $j$ is substituted with zeros, and the $\ell$-th element is replaced by $(-1)^{j+\ell}$.

\begin{lemma}\label{LemmaOrderKjl}
   Let  $j,\ell$, and $n$ be integers and suppose $1\le j\neq \ell\le n$.
 Then: 
 \begin{enumerate}[itemsep=4pt]
  \item[\namedlabel{LemmaOrderKjlA}{\normalfont{(i)}}]
    $K_jK_\ell K_j=K_\ell K_jK_\ell
    = \Id -\be_j \be_j^T - \be_\ell \be_\ell^T
  + (-1)^{j+\ell}
  \left(\be_j \be_\ell^T + \be_l\be_j^T\right)$;
    
  \item[\namedlabel{LemmaOrderKjlB}{\normalfont{(ii)}}]
    $(K_jK_\ell)^3 = \Id$.
\end{enumerate} 
\end{lemma}
\begin{proof}
By Lemma~\ref{LemmaFormulaKjl} and~\eqref{eqDefKj}, we have:
\begin{equation}\label{eqKjlKj1}
  \begin{split}
  K_jK_\ell K_j &=
    \left(\Id- \be_j \be_j^T - \be_\ell \be_\ell^T
    + \be_\ell\br_\ell 
    + (-1)^{j+\ell}\be_j \be_\ell^T\right)
    \left(\Id - \be_j \be_j^T 
  + \be_j\br_j\right)\\
  &=
  \Id- \be_j \be_j^T - \be_\ell \be_\ell^T
    + \be_\ell\br_\ell 
    + (-1)^{j+\ell}\be_j \be_\ell^T\\
   &\quad - \be_j \be_j^T + \be_j \be_j^T \be_j \be_j^T  
   + \be_\ell \be_\ell^T \be_j \be_j^T 
    - \be_\ell\br_\ell  \be_j \be_j^T 
    + (-1)^{j+\ell+1}\be_j \be_\ell^T \be_j \be_j^T \\
   &\quad +\be_j\br_j - \be_j \be_j^T\be_j\br_j 
   - \be_\ell \be_\ell^T\be_j\br_j
    + \be_\ell\br_\ell \be_j\br_j
    + (-1)^{j+\ell}\be_j \be_\ell^T\be_j\br_j.
  \end{split}    
\end{equation}
Since $\be_j^T \be_j=(1)$ by~Remark~\ref{Remark_G}.\ref{Remark_GA}, and 
$\be_\ell^T \be_j = (0)$ by~Remark~\ref{Remark_G}.\ref{Remark_GD},
expression~\eqref{eqKjlKj1} reduces to
\begin{equation*}
  \begin{split}
  K_jK_\ell K_j &=
  \Id- \be_j \be_j^T - \be_\ell \be_\ell^T
    + \be_\ell\br_\ell 
    + (-1)^{j+\ell}\be_j \be_\ell^T\\
   &\quad 
    - \be_\ell\br_\ell  \be_j \be_j^T 
    + \be_\ell\br_\ell \be_j\br_j.
  \end{split}    
\end{equation*}
Here we observe that $\be_\ell\br_\ell  \be_j \be_j^T$ 
is the zero matrix with the $j$-th element on line $l$
replaced by $(-1)^{l+j-1}$, so that it coincides with 
$(-1)^{l+j-1}\be_l\be_j^T$. Therefore
\begin{equation*}
  \begin{split}
  K_jK_\ell K_j &=
  \Id -\be_j \be_j^T - \be_\ell \be_\ell^T
  + (-1)^{j+\ell}
  \left(\be_j \be_\ell^T + \be_l\be_j^T\right)
    + \be_\ell\br_\ell \left(\Id+\be_j\br_j\right).
  \end{split}    
\end{equation*}
In the previous equality, the last term cancels out according to Remark~\ref{Remark_G}.\ref{Remark_GE} and, since it is invariant under the interchange of $j$ with $l$, the equalities in~\ref{LemmaOrderKjlA} follow.

In order to establish~\ref{LemmaOrderKjlB},
we start with the equality $K_jK_\ell K_j=K_\ell K_jK_\ell$ and multiply it successively on the right 
by $K_l$, $K_j$, and $K_l$, noting that the involution shown in Lemma~\ref{LemmaInvolutionK} 
leads to the right-hand side of the outcome being equal to $\Id$.

This completes the proof of Lemma~\ref{LemmaOrderKjl}.
\end{proof}
Here are three examples of matrices yielded by the products 
in Lemma~\ref{LemmaOrderKjl}.\ref{LemmaOrderKjlA} in dimension~$n=4$:
\begin{align*}
  K_1K_2K_1 =  
  \begin{pmatrix}
    0 & -1 & 0 & 0 \\
    -1 & 0 & 0 & 0 \\
    0 & 0 & 1 & 0 \\
    0 & 0 & 0 & 1   
  \end{pmatrix};\ \ 
    K_3K_1K_3 =  
  \begin{pmatrix}
    0 & 0 & 1 & 0 \\
    0 & 1 & 0 & 0 \\
    1 & 0 & 0 & 0 \\
    0 & 0 & 0 & 1
  \end{pmatrix};\ \ 
      K_2K_4K_2 =  
  \begin{pmatrix}
    1 & 0 & 0 & 0 \\
    0 & 0 & 0 & 1 \\
    0 & 0 & 1 & 0 \\
    0 & 1 & 0 & 0 
  \end{pmatrix}.
\end{align*}

The lemmas that follow provide the pattern of matrices generated by ordered products.
To illustrate, if $n=4$ we have:
\begin{align*}
  K_1K_2K_3 =  
  \begin{pmatrix}
    0 & -1 & 0 & 0 \\
    0 & 0 & -1 & 0 \\
    -1 & 1 & -1 & 1 \\
    0 & 0 & 0 & 1 
  \end{pmatrix}\ \ 
    \text{ and }\ \ 
      K_4K_3K_2 =  
  \begin{pmatrix}
    1 & 0 & 0 & 0 \\
    1 & -1 & 1 & -1 \\
    0 & -1 & 0 & 0 \\
    0 & 0 & -1 & 0
  \end{pmatrix}.
\end{align*}
Thus the product $K_1K_2\cdots K_s$ acts as if 
it makes the following changes to the identity matrix:
line $s$ is replaced with the alternating line $\br_s$, 
in the lines above the elements on the main diagonal are replaced with zeros, 
and the adjacent elements to their right are replaced with $(-1)$'s.

The pattern is adapted for the product $K_nK_{n-1}\cdots K_{n-s+1}$.
In this case, the result makes the following modifications to the identity matrix:
line $n-s+1$ is replaced with the alternating line $\br_{n-s+1}$,
the diagonal elements on the lines below it are zeroed out, 
and the adjacent elements to their left are replaced with $(-1)$'s.

\begin{lemma}\label{LemmaOrderedProducts}
Let  $s$ and $n$ be integers and suppose $2\le s\le n$.
 Then: 
\begin{align}
   K_1K_2\cdots K_s &=
   \Id 
-\sum_{j=1}^{s}\be_j\be_j^T
-\sum_{j=1}^{s-1}\be_j\be_{j+1}^T
+ \be_s\br_s
   \label{eqProductUP}\\
   \shortintertext{and}
K_{n}K_{n-1}\cdots K_{n-s+1} &= 
\Id -\sum_{j=1}^{s}\be_{n-j+1}\be_{n-j+1}^T 
    -\sum_{j=1}^{s-1}\be_{n-j+1}\be_{n-j}^T + \be_{n-s+1}\br_{n-s+1}.
  \label{eqProductDown}
\end{align}
\end{lemma}
\begin{proof}
We prove the results by induction on $s$.
If $s=2$, the formulas follow from Lemma~\ref{LemmaFormulaKjl}.
Let us assume that~\eqref{eqProductUP} holds for some fixed number of factors
\mbox{$s$ with $2\le s\le n-1$}.
Then
\begin{align*}
   K_1\cdots K_s\cdot K_{s+1} &=
   \left(\Id 
-\sum_{j=1}^{s}\be_j\be_j^T
-\sum_{j=1}^{s-1}\be_j\be_{j+1}^T
+ \be_s\br_s\right)
   \left(\Id - \be_{s+1} \be_{s+1}^T 
  + \be_{s+1}\br_{s+1}\right)\\
   &=\Id 
-\sum_{j=1}^{s}\be_j\be_j^T
-\sum_{j=1}^{s-1}\be_j\be_{j+1}^T
+ \be_s\br_s\\
&\quad 
 - \be_{s+1} \be_{s+1}^T
 + \sum_{j=1}^{s}\be_j(\be_j^T  \be_{s+1}) \be_{s+1}^T
 + \sum_{j=1}^{s-1}\be_j(\be_{j+1}^T  \be_{s+1}) \be_{s+1}^T
 - \be_s\br_s  \be_{s+1} \be_{s+1}^T\\
&\quad 
 + \be_{s+1} \br_{s+1}
 - \sum_{j=1}^{s}\be_j(\be_j^T  \be_{s+1}) \br_{s+1}
 - \sum_{j=1}^{s-1}\be_j(\be_{j+1}^T  \be_{s+1}) \br_{s+1}
 + \be_s\br_s  \be_{s+1} \br_{s+1}.
\end{align*}
Since, according to Remark~\ref{Remark_G}.\ref{Remark_GD}, 
the multiple sums from the last two lines above cancel out, it follows
\begin{align*}
   K_1\cdots K_s\cdot K_{s+1} &=
\Id-\sum_{j=1}^{s+1}\be_j\be_j^T
-\sum_{j=1}^{s}\be_j\be_{j+1}^T
 + \be_{s+1} \br_{s+1}
\\
&\quad 
 + \be_{s+1}\be_{s+1}^T +  \be_s\be_{s+1}^T 
\\
&\quad 
 - \be_{s+1} \be_{s+1}^T
 - \be_s\br_s  \be_{s+1} \be_{s+1}^T\\
&\quad 
 + \be_s\br_s + \be_s\br_s  \be_{s+1} \br_{s+1},
\end{align*}
where, on the first line, we completed the sums of the formula 
to be proved for the order~\mbox{$s+1$},
and also made the necessary interchange for this of the position of the terms 
$\be_s\br_s$ and $\be_{s+1}\br_{s+1}$.

It remains to check that the sum of the terms on the last three lines above 
cancels each other out. Indeed, by grouping them conveniently, their sum equals
\begin{align*}
 (\be_{s+1}\be_{s+1}^T-\be_{s+1}\be_{s+1}^T) 
 +  (\be_s\be_{s+1}^T - \be_s\br_s  \be_{s+1} \be_{s+1}^T)
 + (\be_s\br_s + \be_s\br_s  \be_{s+1} \br_{s+1}) = \bzero,
\end{align*}
because the sum in the second parenthesis cancels out since, 
according to Remark~\ref{Remark_G}.\ref{Remark_GA}, 
we have $\br_s\be_{s+1}=(-1)^{s+(s+1)-1}=(1)$,
and the sum in the last parentheses cancels out since, 
according to Remark~\ref{Remark_G}.\ref{Remark_GE}, 
we have $\be_s\br_s\be_{s+1}\br_{s+1}=-\be_s\br_s$.

The proof of formula~\eqref{eqProductDown} follows along the same lines.
This completes the proof of Lemma~\ref{LemmaOrderedProducts}. 
\end{proof}

A revised pattern is created by the reverted products. 
When $n=4$ and $s=3$, it produces:
\begin{align*}
  K_3K_2K_1 =  
  \begin{pmatrix}
    -1 & 1 & -1 & 1 \\
    -1 & 0 & 0 & 0 \\
    0 & -1 & 0 & 0 \\
    0 & 0 & 0 & 1
  \end{pmatrix}\ \ 
    \text{ and }\ \ 
      K_2K_3K_4 =  
  \begin{pmatrix}
    1 & 0 & 0 & 0 \\
    0 & 0 & -1 & 0 \\
    0 & 0 & 0 & -1 \\
    1 & -1 & 1 & -1
  \end{pmatrix},
\end{align*}
and the general result follows 
by adjusting in accordance with the proof of Lemma~\ref{LemmaOrderedProducts}.  
\begin{lemma}\label{LemmaOrderedProductsReversed}
Let  $s$ and $n$ be integers and suppose $2\le s\le n$.
 Then: 
\begin{align}
   K_sK_{s-1}\cdots K_1 &=
   \Id 
-\sum_{j=1}^{s}\be_j\be_j^T
-\sum_{j=2}^{s}\be_j\be_{j-1}^T
+ \be_1\br_1
   \label{eqProductUPReversed}\\
   \intertext{and}
K_{n-s+1}\cdots K_{n-1}K_{n} &= 
\Id -\sum_{j=1}^{s}\be_{n-j+1}\be_{n-j+1}^T 
    -\sum_{j=2}^{s}\be_{n-j+1}\be_{n-j+2}^T + \be_{n}\br_{n}.
  \label{eqProductDownReversed}
\end{align}
\end{lemma}

\begin{theorem}\label{PropositionPodDI}
  We have:
\begin{align*}
K_nK_{n-1}\cdots K_2K_1 &=
\begin{pmatrix}
-1 & 1 & -1 & 1 & \cdots & (-1)^{n-1} & (-1)^{n} \\
-1 & 0 & 0 & 0 & \cdots & 0 & 0 \\
0 & -1 & 0 & 0 & \cdots & 0 & 0 \\
0 & 0 & -1 & 0 & \cdots & 0 & 0 \\
\vdots & \vdots & \vdots & \vdots & \ddots & \vdots & \vdots \\
0 & 0 & 0 & 0 & \cdots & 0 & 0 \\
0 & 0 & 0 & 0 & \cdots & -1 & 0 
\end{pmatrix},\\[8pt]
K_1 K_2\cdots K_{n-1}K_n &= 
\arraycolsep=-1pt 
\begin{pNiceMatrix}[columns-width = 44pt]
0 & -1 & 0 & 0 & \cdots & 0 & 0 \\
0 & 0 & -1 & 0 & \cdots & 0 & 0 \\
0 & 0 & 0 & -1 & \cdots & 0 & 0 \\
\vdots & \vdots & \vdots & \vdots & \ddots & \vdots & \vdots \\
0 & 0 & 0 & 0 & \cdots & 0 & -1 \\
(-1)^{n} & (-1)^{n+1} & (-1)^{n+2} & (-1)^{n+3} & \cdots & (-1)^{2n-1} & (-1)^{2n} 
\end{pNiceMatrix},
\end{align*}
\begin{align}\label{eqProdAllDI}
(K_nK_{n-1}\cdots K_2K_1)^{n+1} &=
(K_1 K_2\cdots K_{n-1}K_n)^{n+1} = \Id,
\end{align}
and $n+1$ 
is the smallest
$N\ge 1$ for which 
\mbox{$(K_nK_{n-1}\cdots K_2K_1)^{N} =
(K_1 K_2\cdots K_{n-1}K_n)^{N} = \Id$.}
\end{theorem}
\begin{proof}\renewcommand{\qedsymbol}{} 
The formulas for the `down' and `up' products follow from Lemmas~\ref{LemmaOrderedProducts} and~\ref{LemmaOrderedProductsReversed}.
To prove the formula for the order, we write
\begin{align}\label{eqKseS}
  K_nK_{n-1}\cdots K_2K_1 = \be_1\br_1 - S, 
\end{align}
where $S$ is the sub-diagonal matrix:
\begin{align*}
S &:= \sum_{j=1}^{n-1} \be_{j+1}\be_j^T =
\begin{pmatrix}
0 & 0 & \cdots & 0 & 0 \\
1 & 0 & \cdots & 0 & 0 \\
0 & 1 &  \cdots & 0 & 0 \\
\vdots & \vdots & \vdots & \vdots & \ddots  \\
0 & 0 & \cdots & 1 & 0 
\end{pmatrix}.
\end{align*}
Note that, by induction, we obtain: 
\begin{align}\label{eqPowers-ejrj}
  \left(\be_j\br_j\right)^k
  = (-1)^{k+1} \be_j\br_j, \text{ for $1\le j,k\le n$} 
\end{align}
and the powers of $S$ are the sub-diagonal matrices:
\begin{align}\label{eqPowers-S}
  S^k   = \sum_{\ell=k+1}^{n} \be_{\ell}\be_{\ell-k}^T, 
  \text{ for $1\le k\le n-1$ and $S^n = \bzero$.} 
\end{align}
The next natural step involves substituting~\eqref{eqPowers-ejrj} 
and~\eqref{eqPowers-S} into~\eqref{eqKseS}, raising the result to the power of $n+1$, and analyzing the result. 
This can be done with all the complications caused by the non-commutativity of the multiplications involved, but for a quicker approach, we will complete the proof of 
Theorem~\ref{PropositionPodDI} in Section~\ref{SectionProofPropositionPodDI}.
\end{proof}

In order to show the formula for calculating more general products, we need to introduce the following notations.
Let $\bj = \List{j_1,\dots,j_s}\in\{1,2,\dots,n\}^s$ 
denote an ordered set (also viewed as an~\mbox{$s$-tuple}) of indices between~$1$ and $n$, and let  
$\cD(\bj)$ denote the set of pairs of adjacent elements of $\bj$, that is,
\begin{align*}
   \cD(\bj) := \{
   (j_1,j_2),(j_2,j_3),\dots,(j_{s-1},j_s)
   \}.
\end{align*}
Note that $\cD(\bj)$ has $s-1$ elements and it is empty if $s=1$.

\begin{theorem}\label{TheoremOrderMany}
Let $s$ and $n$ be positive integers and suppose $1\le s\le n$.
Let $\bj = \List{j_1,\dots,j_s}$ be
an $s$-tuple with distinct components from $\{1,2,\dots,n\}$.
Then
\begin{align}\label{eqProductlMany}
   K_{j_1}K_{j_2}\cdots K_{j_s}
    =\Id 
-\sum_{j\in\bj}\be_j\be_j^T
+\sum_{(j,\ell)\in \cD(\bj)}(-1)^{j+\ell}\be_j\be_{\ell}^T
+ \be_{j_s}\br_{j_s}.
\end{align}
\end{theorem}
\begin{proof}
  We start by observing that if $s=1$, then $\bj=\List{j}$ has only one element, 
which makes the second sum in~\eqref{eqProductlMany} empty, and the formula reduces 
to definition~\eqref{eqDefKj}.
If $s=2$, then \mbox{$\bj=\List{j,\ell}$}, and~\eqref{eqProductlMany} reduces to
formula~\eqref{eqLemmaFormulaKjl}.
If $\bj=\List{1,2,\dots,s}$
or \mbox{$\bj=\List{n,n-1,\dots,n-s+1}$}, then~\eqref{eqProductlMany} reduces to 
formula~\eqref{eqProductUP} or~\eqref{eqProductDown}.
Likewise, when the tuples are the reversed 
\mbox{$\bj=\List{s, s-1,\dots,1}$}
or $\bj=\List{n-s+1,\dots,n-1,n}$, 
then~\eqref{eqProductUPReversed} and~\eqref{eqProductDownReversed} 
are also covered by formula~\eqref{eqProductlMany}.

In order to prove~\eqref{eqProductlMany} on $s$,
once the initial cases have been verified, let us assume 
that it holds for a fixed $s$. 
Let $\bj'=\List{j_1,\dots,j_s,k}$, 
where $k$ is distinct from all $j_1, j_2,\dots, j_s$.
Then
\begin{align*}
   K_{j_1}K_{j_2}\cdots K_{j_s} K_{k}
    & =
    \left(\Id 
-\sum_{j\in\bj}\be_j\be_j^T
+\sum_{(j,\ell)\in \cD(\bj)}(-1)^{j+\ell}\be_j\be_{\ell}^T
+ \be_{j_s}\br_{j_s}\right)
\left(\Id - \be_{k} \be_{k}^T 
  + \be_{k}\br_{k}\right)\\
   &=\Id 
-\sum_{j\in\bj}\be_j\be_j^T
+\sum_{(j,\ell)\in \cD(\bj)}(-1)^{j+\ell}\be_j\be_{\ell}^T
+ \be_{j_s}\br_{j_s}\\
&\quad 
 - \be_{k} \be_{k}^T 
+\sum_{j\in\bj}\be_j(\be_j^T \be_{k}) \be_{k}^T 
-\sum_{(j,\ell)\in \cD(\bj)}(-1)^{j+\ell}\be_j(\be_{\ell}^T \be_{k}) \be_{k}^T 
- \be_{j_s}\br_{j_s} \be_{k} \be_{k}^T \\
&\quad 
 + \be_{k}\br_{k}
-\sum_{j\in\bj}\be_j(\be_j^T \be_{k})\br_{k}
+\sum_{(j,\ell)\in \cD(\bj)}(-1)^{j+\ell}\be_j(\be_{\ell}^T \be_{k})\br_{k}
+ \be_{j_s}\br_{j_s}\be_{k}\br_{k}.
\end{align*}
Since $k$ is distinct from all the components of $\bj$, the multiple sums from the last two lines above cancel out, according to Remark~\ref{Remark_G}.\ref{Remark_GD}.
Upon completing the expressions and rearranging the terms, it follows
\begin{align*}
   K_{j_1}K_{j_2}\cdots K_{j_s} K_{k}
    & = \Id 
-\sum_{j\in\bj'}\be_j\be_j^T
+\sum_{(j,\ell)\in \cD(\bj')}(-1)^{j+\ell}\be_j\be_{\ell}^T
+ \be_{k}\br_{k}\\
&\quad 
+ \be_{k} \be_{k}^T + (-1)^{j_s+k+1}\be_{j_s}\be_{k}^T\\
&\quad 
- \be_{k} \be_{k}^T - \be_{j_s}\br_{j_s} \be_{k} \be_{k}^T \\
&\quad
+ \be_{j_s}\br_{j_s} + \be_{j_s}\br_{j_s}\be_{k}\br_{k}.
\end{align*}

To conclude the proof, we only need to show that the following sum cancels out:
\begin{align*}
 \Sigma := (\be_{k} \be_{k}^T-\be_{k} \be_{k}^T) 
 +  \left((-1)^{j_s+k+1}\be_{j_s}\be_{k}^T 
       - \be_{j_s}\br_{j_s} \be_{k} \be_{k}^T\right)
 + (\be_{j_s}\br_{j_s} + \be_{j_s}\br_{j_s}\be_{k}\br_{k}).
\end{align*}
By Remark~\ref{Remark_G}.\ref{Remark_GA}, 
we know that $\br_{j_s} \be_{k}=(-1)^{j_s+k-1}=(1)$, so that the second
parentheses cancel out.
Further, by Remark~\ref{Remark_G}.\ref{Remark_GE}, 
we know $ \be_{j_s}\br_{j_s}\be_{k}\br_{k}=-\be_{j_s}\br_{j_s}$,
which implies that the third parentheses also cancel out.
In conclusion, $\Sigma=\bzero$.

This concludes the proof of Theorem~\ref{TheoremOrderMany}.
\end{proof}

\section{\texorpdfstring{The set of matrices $\cM(n)$}{The set of matrices M(n)}}

The multiplication of matrices with $K_j$ in general does not commute, 
and the effect is quite different depending on whether the multiplication 
is done on the left or on the right. 
For instance, here are the transformations produced to a random matrix $M$:
\begin{align*}
  M =  
  \begin{pmatrix}
    5 & 2 & 1 & 5 & 3 \\
    5 & 5 & 4 & 4 & 5 \\
    4 & 1 & 1 & 2 & 1 \\
    2 & 4 & 1 & 5 & 4 \\
    2 & 5 & 2 & 3 & 1 
  \end{pmatrix}:\ \ 
    MK_3 =  
  \begin{pmatrix}
    4 & 3 & -1 & 6 & 2 \\
    1 & 9 & -4 & 8 & 1 \\
    3 & 2 & -1 & 3 & 0 \\
    1 & 5 & -1 & 6 & 3 \\
    0 & 7 & -2 & 5 & -1
  \end{pmatrix},\ \ 
      K_3M =  
  \begin{pmatrix}
    5 & 2 & 1 & 5 & 3 \\
    5 & 5 & 4 & 4 & 5 \\
    -4 & 1 & 1 & -1 & 4 \\
    2 & 4 & 1 & 5 & 4 \\
    2 & 5 & 2 & 3 & 1
  \end{pmatrix}.
\end{align*}

We can see that what holds true in general, and follows directly 
from the definition of~$K_j$, is that multiplying 
by $K_j$ on the right disrupts all columns, 
while multiplying by $K_j$ on the left changes line $j$ in $M$ 
while keeping all others intact.

We let $\scrG_n(K):=\langle K_1,K_2,\dots,K_n\rangle$ denote the 
finite group of matrices generated by $K_1,K_2,\dots,K_n$, the operation
being matrix multiplication.

Checking the cardinality of $\scrG_n(K)$ for small $n$ using the properties from Section~\ref{SectionLemmas}, we find:
$\#\scrG_2(K)=6$, $\#\scrG_3(K) = 24$. 
Also, one is led to conclude that each product of $K_j$'s is a matrix that contains only $0$'s and $\pm 1$'s. This suggests that
$\#\scrG_2(K)$ is finite and contains at most $3^{n^2}$ elements.

 Moreover, a more precise identification of these matrices seems plausible. We observed that matrices of the following type appear as products of matrices in $\scr{G}_n(K)$.

Let $\sigma$ be a permutation in $S_n$, $h\in\{1,2,\dots,n\}$,
and $\epsilon\in\{0,1\}$. Then we let $M(\sigma,h,\epsilon)$
be the $n\times n$ matrix whose entries are defined by:
\begin{align*}
  m_{j,k}:=
  \begin{cases}
    (-1)^{j+k}, & \text{ if $\epsilon = 0$ and $k = \sigma(j)$;}\\[3pt]
    0, & \text{ if $\epsilon = 0$ and $k\neq \sigma(j)$;}\\[3pt] (-1)^{j+k}, & \text{ if $\epsilon = 1$, $j\neq h$, and $k = \sigma(j)$;}\\[3pt]
    0, & \text{ if $\epsilon = 1$, $j\neq h$,  and $k\neq \sigma(j)$;}\\[3pt] 
    (-1)^{j+k-1}, & \text{ if $\epsilon = 1$, $j = h$.}
  \end{cases}
\end{align*}
Piecing together the non-zero entries,  
$M(\sigma,h,\epsilon)$ is written as:
\begin{align}\label{eqDefMsih}
  M(\sigma,h,\epsilon):=
   \sum_{j=1}^{n}(-1)^{j+\sigma(j)}\be_j\be_{\sigma(j)}^T
+\epsilon\left(
 (-1)^{h+\sigma(h)+1}\be_h\be_{\sigma(h)}^T
 + \sum_{k=1}^{n}(-1)^{h+k-1}\be_h\be_{k}^T
\right).
\end{align}
We let $\cM(n)$ denote the set of all these matrices:
\begin{align*}
    \cM(n) :=\big\{  M(\sigma,h,\epsilon) :
    \sigma\in S_n,\ 1\le h \le n,\ \epsilon\in\{0,1\}\big\}.
\end{align*}
We notice the inclusion 
$\{K_1,K_2,\dots,K_n\}\subset\cM(n)$, since
$K_h = M(\id,h,1)$, where $\id$ is the identity permutation in $S_n$.
In the definition~\eqref{eqDefMsih} of $M(\sigma,h,\epsilon)$, the terms play the following roles. 
Starting with the zero matrix of order $n\times n$,
the first sum places $\pm 1$'s at position~$jk$ 
in such a way that exactly one non-zero element appears 
on each line and on each column. 
Then, if $\epsilon=1$, the first term in the parentheses that 
follows cancels the non-zero component just placed on row $h$,
and the subsequent sum replaces line $h$ with the alternating 
sequence of $\pm 1$'s, starting with $(-1)^h$.
Anyhow, $\cM(n)\subset GL(n,\ZZ)$ and
with a hint of intuition, $M(\sigma,h,\epsilon)$ 
can also be written as
\begin{align}\label{eqDefMsih2}
  M(\sigma,h,\epsilon):= M(\sigma,h,0)
  + \epsilon\left((-1)^{h+\sigma(h)+1}
 \be_h\be_{\sigma(h)}^T
 + \be_h\br_{h}
\right).
\end{align}

Further analysis of other properties is the object of the next theorem.
\begin{theorem}\label{LemmaMsihMultiplication}
For any permutations $\sigma,\tau\in S_n$ and any integers 
$ u,v,w$ satisfying $1\le u,v,w\le n$, we have:
\begin{align}
  M(\sigma,u,0)M(\tau,v,0) &= M(\tau\sigma ,w,0),
  \label{eqMultiplication00A}\\
  M(\sigma, u, 1)M(\tau, v, 0) &= M(\tau\sigma, u, 1),
  \label{eqMultiplication00B}\\
M(\sigma, u, 0)M(\tau, v, 1) &= M(\tau\sigma, \sigma^{-1}(v), 1),
\label{eqMultiplication00C}\\
M(\sigma, u, 1)M(\tau, v, 1) &= 
\begin{cases}
    M( \tau\sigma, w, 0), &
    \text{ if } \sigma^{-1}(v) = u;\\
    M(\eta, \sigma^{-1}(v), 1), &
    \text{ if } \sigma^{-1}(v)\neq u,
\end{cases}
\label{eqMultiplication00D}
\end{align}
where 
\begin{align}\label{eqeta}
\eta(j) = 
    \begin{cases}    
    \tau\sigma(j),  & \text{if } 
    j \in\{1,2,\dots,n\}\setminus\{u,\sigma^{-1}(v)\}; \\   
    \tau\sigma(u),  & \text{if } j = \sigma^{-1}(v);\\
    \tau\sigma(\sigma^{-1}(v)),  & \text{if } j = u.
     \end{cases}
\end{align}
\end{theorem}

\begin{proof}
\textbf{\texttt{1.}} We have
\begin{align*}
  M(\sigma,u,0)M(\tau,v,0) &=
 \sum_{j=1}^{n} \sum_{\ell=1}^{n}
 (-1)^{(j+\sigma(j))+(\ell+\tau(\ell))}
 \be_j(\be_{\sigma(j)}^T \be_\ell)\be_{\tau(\ell)}^T.
\end{align*}
Since, by Remark~\ref{Remark_G}~\ref{Remark_GD}, we know that 
$\be_{\sigma(j)}^T \be_\ell=(0)$ unless $\ell=\sigma(j)$,
in which case $\be_{\sigma(j)}^T \be_\ell=(1)$,
we derive
\begin{align*}
  M(\sigma,u,0)M(\tau,v,0) &=
 \sum_{j=1}^{n} 
 (-1)^{(j+\sigma(j))+(\sigma(j)+\tau(\sigma(j)))}
 \be_j(\be_{\sigma(j)}^T \be_{\sigma(j)})\be_{\tau(\sigma(j))}^T\\
  &=
 \sum_{j=1}^{n} 
 (-1)^{j+\tau(\sigma(j))}
 \be_j\be_{\tau(\sigma(j))}^T,
\end{align*}
which verifies formula~\eqref{eqMultiplication00A}.

\medskip
\noindent
\textbf{\texttt{2.}} Next, by~\eqref{eqDefMsih2} and~\eqref{eqMultiplication00A}, we derive:
\begin{equation}\label{eqM12}
\begin{split}
  M(\sigma,u,1)M(\tau,v,0) &= 
  \left(M(\sigma,u,0) +(-1)^{u+\sigma(u)+1} \be_u\be_{\sigma(u)}^T + \be_u\br_u\right) 
  M(\tau,v,0)\\
  & = M(\tau\sigma,u,0)  -(-1)^{u+\sigma(u)} \be_u\be_{\sigma(u)}^T M(\tau,v,0)
  + \be_u\br_u M(\tau,v,0),
  \end{split}
\end{equation}
as in~\eqref{eqMultiplication00A}, the parameter $w$ has no impact,
it can be any, specifically, it can be chosen as $u$.

Further, we have
\begin{equation}\label{eqM12a}
  \begin{split}  (-1)^{u+\sigma(u)+1}\be_u\be_{\sigma(u)}^T M(\tau,v,0) &= 
  \sum_{\ell=1}^{n}(-1)^{u+\sigma(u)+\ell+\tau(\ell)+1}\be_u(\be_{\sigma(u)}^T\be_\ell)\be_{\tau(\ell)}^T\\
   &= 
   (-1)^{u+\tau(\sigma(u))+1} \be_u \be_{\tau(\sigma(u))}^T.
  \end{split}
\end{equation}
And the last term in~\eqref{eqM12} is:
\begin{equation}\label{eqM12b}
  \begin{split}
   \be_u\br_u M(\tau,v,0) &= 
  \sum_{\ell=1}^{n}(-1)^{\ell+\tau(\ell)}\be_u(\br_u\be_\ell)\be_{\tau(\ell)}^T\\
   &=  \sum_{\ell=1}^{n}(-1)^{(\ell+\tau(\ell))+(u+\ell-1)}
   \be_u\be_{\tau(\ell)}^T\\
   &=  \sum_{\ell=1}^{n}
   (-1)^{u+\tau(\ell)-1}
   \be_u\be_{\tau(\ell)}^T\\
   &=  \sum_{k=1}^{n}
   (-1)^{u+k-1}
   \be_u\be_{k}^T\\
   &= \be_u\br_u.
  \end{split}
\end{equation}
Then, upon inserting~\eqref{eqM12a} and~\eqref{eqM12b} 
into~\eqref{eqM12}, we find~\eqref{eqMultiplication00B}.

\medskip
\noindent
\textbf{\texttt{3.}}In order to establish~\eqref{eqMultiplication00C}, the first part of the calculation is similar:
\begin{equation}\label{eqM13}
\begin{split}
  M(\sigma,u,0)M(\tau,v,1) &= 
  M(\sigma,u,0)
  \left( M(\tau,v,0)+(-1)^{v+\tau(v)+1} \be_v\be_{\tau(v)}^T + \be_v\br_v\right) \\
  & = M(\tau\sigma,\sigma^{-1}(v),0)  \\
  &\quad +(-1)^{v+\tau(v)+1} M(\sigma,u,0)\be_v\be_{\tau(v)}^T 
  + M(\sigma,u,0)\be_v\br_v.
  \end{split}
\end{equation}
Here, we have again used the observation that in~\eqref{eqMultiplication00A}, the parameter $w$ has no influence,
it can be any, in particular, it can be taken as $\sigma^{-1}(v)$.

The first term on the last line of the equation~\eqref{eqM13} is:
\begin{equation}\label{eqM13a}
  \begin{split}  
  (-1)^{v+\tau(v)+1}M(\sigma,u,0)\be_v\be_{\tau(v)}^T  
  &= 
  \sum_{j=1}^{n}(-1)^{v+\tau(v)+j+\sigma(j)+1}\be_j(\be_{\sigma(j)}^T\be_v)\be_{\tau(v)}^T\\
   &= 
   (-1)^{\tau(v)+\sigma^{-1}(v)+1} \be_{\sigma^{-1}(v)} \be_{\tau(v)}^T,
  \end{split}
\end{equation}
and the second is:
\begin{equation}\label{eqM13b}
  \begin{split}  
  M(\sigma,u,0)\be_v\br_v  
  &= 
  \sum_{j=1}^{n}(-1)^{j+\sigma(j)}\be_j(\be_{\sigma(j)}^T\be_v)\br_v \\
   &= (-1)^{\sigma^{-1}(v)+v}\be_{\sigma^{-1}(v)}\br_v.
  \end{split}
\end{equation}
Now, let us observe that the first term in the parenthesis on the right-hand side of formula~\eqref{eqDefMsih2} 
for $M(\tau\sigma,\sigma^{-1}(v),1)$ equals
\begin{equation}\label{eqM13aa}
  \begin{split}  
  (-1)^{\sigma^{-1}(v)+\tau\sigma(\sigma^{-1}(v))+1}
  \be_{\sigma^{-1}(v)}\be_{\tau\sigma(\sigma^{-1}(v))}^T  
  &= 
  (-1)^{\sigma^{-1}(v)+\tau(v)+1}
  \be_{\sigma^{-1}(v)}\be_{\tau(v)}^T, 
  \end{split}
\end{equation}
which coincides with the outcome in~\eqref{eqM13a}.
Besides, the second term in the parenthesis on the right-hand side of formula~\eqref{eqDefMsih2} 
for $M(\tau\sigma,\sigma^{-1}(v),1)$ equals
\begin{equation}\label{eqM13bb}
  \begin{split}  
  \be_{\sigma^{-1}(v)}\br_{\sigma^{-1}(v)}
  & =
 \be_{\sigma^{-1}(v)}
 \left((-1)^{\sigma^{-1}(v)+v}\br_{v}\right)
=(-1)^{\sigma^{-1}(v)+v}\be_{\sigma^{-1}(v)}\br_v,
  \end{split}
\end{equation}
that is, the same result as the one in~\eqref{eqM13b}. 
Therefore, by replacing the results from~\eqref{eqM13aa} into~\eqref{eqM13a} 
and from~\eqref{eqM13bb} into~\eqref{eqM13b} 
in place of those from~\eqref{eqM13}, we obtain exactly formula~\eqref{eqMultiplication00C}.

\medskip
\noindent
\textbf{\texttt{4.}}
To sum up, in order to prove~\eqref{eqMultiplication00D}, we employ~\eqref{eqDefMsih2} with the purpose
to make use of~\eqref{eqMultiplication00A}-\eqref{eqMultiplication00C}.
Proceeding, we have:
\begin{equation}\label{eqM13A}
\begin{split}
  M(\sigma,u,1)M(\tau,v,1) &= 
  M(\sigma,u,1)
  \left( M(\tau,v,0)+(-1)^{v+\tau(v)+1} \be_v\be_{\tau(v)}^T + \be_v\br_v\right) \\
  & = M(\tau\sigma,u,1)  \\
  &\quad +(-1)^{v+\tau(v)+1} M(\sigma,u,1)\be_v\be_{\tau(v)}^T 
  + M(\sigma,u,1)\be_v\br_v.
  \end{split}
\end{equation}

 Using~\eqref{eqDefMsih2}, the first term on the last line above equals
\begin{equation}\label{eqM14a}
\begin{split}
 (-1)^{v+\tau(v)+1} M(\sigma,u,1)\be_v\be_{\tau(v)}^T &=
 (-1)^{v+\tau(v)+1}M(\sigma,u,0)\be_v\be_{\tau(v)}^T\\
 &\quad
 +(-1)^{v+\tau(v)+u+\sigma(u)} \be_u(\be_{\sigma(u)}^T\be_v)\be_{\tau(v)}^T\\
  &\quad
  -(-1)^{v+\tau(v)}\be_u(\br_{u}\be_v)\be_{\tau(v)}^T,
  \end{split}
\end{equation}
while the second term on the last line is 
\begin{equation}\label{eqM14b}
\begin{split}
 M(\sigma,u,1)\be_v\br_v &=
  M(\sigma,u,0)\be_v\br_v 
  + (-1)^{u+\sigma(u)+1}  \be_u(\be_{\sigma(u)}^T \be_v)\br_v 
 + \be_u(\br_{u} \be_v)\br_v.
  \end{split}
\end{equation}

Since the final result depends on the value of the associated factors 
in~\eqref{eqM14a} and~\eqref{eqM14b}, 
we need to separate the discussion into two cases: when $u=\sigma^{-1}(v)$ 
and when $u\neq \sigma^{-1}(v)$.

\medskip
\noindent
\textbf{\texttt{4.a.}} The case \fbox{$u=\sigma^{-1}(v)$}.

Employing~\eqref{eqM13a}, relation~\eqref{eqM14a} reduces to
\begin{equation*}
\begin{split}
 (-1)^{v+\tau(v)+1} M(\sigma,u,1)\be_v\be_{\tau(v)}^T &=
 (-1)^{\tau(v)+\sigma^{-1}(v)+1} \be_{\sigma^{-1}(v)} \be_{\tau(v)}^T\\
 &\quad
 +(-1)^{\tau(\sigma(u))+u} \be_u\be_{\tau(\sigma(u))}^T\\
  &\quad
  -(-1)^{(v+\tau(v))+(u-1+v)}\be_u \cdot (1)\cdot \be_{\tau(v)}^T.
  \end{split}
\end{equation*}
Knowing that $v = \sigma(u)$, this is an algebraic sum of three matrices identical 
to $\be_u \be_{\tau(\sigma(u))}^T$, only with the signs changed. 
Thus, we obtain that when $u=\sigma^{-1}(v)$, 
the first term on the last line of~\eqref{eqM14a} simplifies to:
\begin{equation}\label{eqM14aa}
\begin{split}
 (-1)^{v+\tau(v)+1} M(\sigma,u,1)\be_v\be_{\tau(v)}^T &=
 \big(
 (-1)^{\tau(v)+u+1}
 + (-1)^{\tau(v)+u}
 + (-1)^{\tau(v)+u}
 \big)
 \be_u \be_{\tau(\sigma(u))}^T\\
  & =  (-1)^{u+\tau(v)}\be_u \be_{\tau(\sigma(u))}^T.
  \end{split}
\end{equation}

Making use of~\eqref{eqM13b} this time, since $v=\sigma(u)$, formula~\eqref{eqM14b} becomes:
\begin{equation}\label{eqM14bb}
\begin{split}
 M(\sigma,u,1)\be_v\br_v &=
  M(\sigma,u,0)\be_v\br_v 
  + (-1)^{u+\sigma(u)+1}  \be_u(\be_{\sigma(u)}^T \be_v)\br_v 
 + \be_u(\br_{u} \be_v)\br_v \\
 &=
 (-1)^{\sigma^{-1}(v)+v}\be_{\sigma^{-1}(v)}\br_v
 + (-1)^{u+\sigma(u)+1}  \be_u\br_{v} 
 + (-1)^{u+v-1}\be_u\br_v\\
  &= 
  \big(
 (-1)^{(u+v)} + (-1)^{u+v+1} + (-1)^{u+v-1)}
 \big)
  \be_u\br_v \\
  &=
  (-1)^{u+v+1}\be_u\br_v.
  \end{split}
\end{equation}

Upon inserting~\eqref{eqM14aa} and~\eqref{eqM14bb} into~\eqref{eqM13A}, we deduce:
\begin{equation}\label{eqM13sigmaveu}
\begin{split}
  M(\sigma,u,1)M(\tau,v,1) &= 
  M(\tau\sigma,u,0)
  + (-1)^{u+\tau(\sigma(u))+1}  \be_u\be_{\tau(\sigma(u))}^T + \be_u\br_{u}\\
  &\quad
  + (-1)^{u+\tau(v)}\be_u \be_{\tau(\sigma(u))}^T\\
  &\quad
  +(-1)^{u+v+1}\be_u\br_v\\
   &=  M(\tau\sigma,w,0).
  \end{split}
\end{equation}
Here, the last four terms cancel out, because $v=\sigma(u)$
and $\be_u\br_{u} = (-1)^{u+v}\be_u\br_{v}$.
As a consequence, formula~\eqref{eqM13sigmaveu} reduces to
\begin{align*}
   M(\sigma,u,1)M(\tau,v,1) =   M(\tau\sigma,w,0),\ \ 
 \text{ if $\sigma(u) = v$, for any $w$,} 
\end{align*}
which completes the proof of the first part of~\eqref{eqMultiplication00D}.

\medskip
\noindent
\textbf{\texttt{4.b.}} The case \fbox{$u\neq \sigma^{-1}(v)$}.
Since $\sigma(u)\neq v$, relation~\eqref{eqM14a} reduces to
\begin{equation}\label{eqM14aM}
\begin{split}
 (-1)^{v+\tau(v)+1} M(\sigma,u,1)\be_v\be_{\tau(v)}^T &=
 (-1)^{v+\tau(v)+1}M(\sigma,u,0)\be_v\be_{\tau(v)}^T\\
  &\quad
  -(-1)^{v+\tau(v)}\be_u(\br_{u}\be_v)\be_{\tau(v)}^T,
  \end{split}
\end{equation}
while~\eqref{eqM14b} reduces to
\begin{equation}\label{eqM14bM}
\begin{split}
 M(\sigma,u,1)\be_v\br_v &=
  M(\sigma,u,0)\be_v\br_v 
 + \be_u(\br_{u} \be_v)\br_v.
  \end{split}
\end{equation}
Then, on using~\eqref{eqM13a} and~\eqref{eqM13b},
relations~\eqref{eqM14aM} and~\eqref{eqM14bM} become:
\begin{equation}\label{eqM14aMb}
\begin{split}
 (-1)^{v+\tau(v)+1} M(\sigma,u,1)\be_v\be_{\tau(v)}^T &=
  (-1)^{\tau(v)+\sigma^{-1}(v)+1} \be_{\sigma^{-1}(v)} \be_{\tau(v)}^T\\
  &\quad
  -(-1)^{v+\tau(v)}\be_u(\br_{u}\be_v)\be_{\tau(v)}^T
  \end{split}
\end{equation}
and
\begin{equation}\label{eqM14bMb}
\begin{split}
 M(\sigma,u,1)\be_v\br_v &=
  (-1)^{\sigma^{-1}(v)+v}\be_{\sigma^{-1}(v)}\br_v
 + \be_u(\br_{u} \be_v)\br_v.
  \end{split}
\end{equation}
Then, inserting~\eqref{eqM14aMb} and~\eqref{eqM14bMb} into~\eqref{eqM13A}, yields:
\begin{equation*}
\begin{split}
  M(\sigma,u,1)M(\tau,v,1)
   & = M(\tau\sigma,u,0) + (-1)^{u+\tau\sigma(u)+1}\be_u\be_{\tau\sigma(u)}^T
   + \be_u\br_{u}\\
   &\quad 
   -(-1)^{\tau(v)+\sigma^{-1}(v)} \be_{\sigma^{-1}(v)} \be_{\tau(v)}^T\\
  &\quad
  -(-1)^{v+\tau(v)}\be_u(\br_{u}\be_v)\be_{\tau(v)}^T\\
  &\quad
  +(-1)^{\sigma^{-1}(v)+v}\be_{\sigma^{-1}(v)}\br_v
 + \be_u(\br_{u} \be_v)\br_v.
  \end{split}
\end{equation*}
Here, on the right-hand side, the second term from the first line 
does not cancel out the term from the third line, 
as it would have done if $v=\sigma(u)$, but
the third term from the first line cancels out 
with the last term from the last line. Thus, the above relation 
only slightly simplifies to:
\begin{equation}\label{eqM13AAAA}
\begin{split}
  M(\sigma,u,1)M(\tau,v,1)
   & = M(\tau\sigma,u,0) + (-1)^{u+\tau\sigma(u)+1}\be_u\be_{\tau\sigma(u)}^T\\
   &\quad 
   -(-1)^{\tau(v)+\sigma^{-1}(v)} \be_{\sigma^{-1}(v)} \be_{\tau(v)}^T\\
  &\quad
  -(-1)^{v+\tau(v) + u-1+v}\be_u\be_{\tau(v)}^T\\
  &\quad
  +(-1)^{\sigma^{-1}(v)+v}\be_{\sigma^{-1}(v)}\br_v.
  \end{split}
\end{equation}
Next, noticing that 
$(-1)^{\sigma^{-1}(v)+v}\be_{\sigma^{-1}(v)}\br_v
= \be_{\sigma^{-1}(v)}\br_{\sigma^{-1}(v)}$, 
we add the other corresponding terms 
to rewrite matrix $M(\tau\sigma,u,0)$ in terms of the kin matrix depending on
the permutation~$\eta$, defined by~\eqref{eqeta}, and~\mbox{$\epsilon=1$}, namely,
\begin{align*}
    M(\eta,\sigma^{-1}(v),1)  
   = M(\eta,\sigma^{-1}(v),0) 
   - (-1)^{\sigma^{-1}(v)+\eta(\sigma^{-1}(v))}
      \be_{\sigma^{-1}(v)} \be_{\eta(\sigma^{-1}(v))}^T
   +  \be_{\sigma^{-1}(v)}\br_{\sigma^{-1}(v)},  
\end{align*}
and then proceed to gather the other terms that
remain unused in the writing of the new matrix,
that are to be subtracted, and the leftover from the second line from~\eqref{eqM13AAAA}.
These terms fall into four categories, as follows:
\begin{enumerate}
      \item 
To be subtracted: the two new terms included in the sum defining 
$M(\eta,\sigma^{-1}(v),0)$ (according to definition~\eqref{eqDefMsih}) and which did not appear in $M(\tau\sigma,u,0)$, namely those for $j = \sigma^{-1}(v)$ and $j = u$;
    \item 
To be added: the terms in the sum defining $M(\tau\sigma,u,0)$ 
that were left out in the sum for $M(\eta,\sigma^{-1}(v),0)$, 
namely those for $j = \sigma^{-1}(v)$ and $j = u$;
    \item 
To be added: the second term on the right-hand side of  
$M(\eta,\sigma^{-1}(v),1)$, according to the expression written above;
    \item 
The three remaining leftover terms from the relation~\eqref{eqM13AAAA},
which we will place on the last three lines of the relation displayed below.
\end{enumerate}
As a result, writing the four categories of terms outlined on successive
lines, relation~\eqref{eqM13AAAA} is equivalent to the following expression:
\begin{equation*}
\begin{split}
  M(\sigma,u,1)M(\tau,v,1) & = 
    M(\eta,\sigma^{-1}(v),1)  \\
    &\quad - (-1)^{\sigma^{-1}(v)+\eta(\sigma^{-1}(v))}
    \be_{\sigma^{-1}(v)} \be_{\eta(\sigma^{-1}(v))}^T  
    -(-1)^{u+\eta(u)}\be_{u} \be_{\eta(u)}^T    \\
   &\quad + (-1)^{\sigma^{-1}(v)+\tau\sigma(\sigma^{-1}(v))}
    \be_{\sigma^{-1}(v)} \be_{\tau\sigma(\sigma^{-1}(v))}^T
    +  (-1)^{u+\tau\sigma(u)}\be_{u} \be_{\tau\sigma(u)}^T
    \\
   &\quad + (-1)^{\sigma^{-1}(v)+\eta(\sigma^{-1}(v))}
      \be_{\sigma^{-1}(v)} \be_{\eta(\sigma^{-1}(v))}^T  
      \\
   &\quad 
   - (-1)^{u+\tau\sigma(u)}\be_u\be_{\tau\sigma(u)}^T\\
  &\quad
   -(-1)^{\tau(v)+\sigma^{-1}(v)} \be_{\sigma^{-1}(v)} \be_{\tau(v)}^T\\
  &\quad
   +(-1)^{u+\tau(v)}\be_u\be_{\tau(v)}^T.
  \end{split}
\end{equation*}
It remains to verify that the last eight terms of the above formula,
which are actually matrices that have exactly one non-zero element each,
cancel out each other. 
Indeed, we find that, except for the opposite signs they have, 
the first term coincides with the fifth, 
the third term coincides with the seventh, 
and the fourth coincides with the sixth.
Finally, we see that the second term coincides with the last
by appealing to the last condition in the definition of $\eta$,
which states that if $j=u$, then 
$\eta(u)= \tau(\sigma(\sigma^{-1})(v)) =\tau(v)$.
In conclusion, 
\begin{align*}
 M(\sigma,u,1)M(\tau,v,1) & =  M(\eta,\sigma^{-1}(v),1),\ \ 
 \text{ if $\sigma(u)\neq v$,}
\end{align*}
which proves the second part of relation~\eqref{eqMultiplication00D}.
This concludes the proof of Theorem~\ref{LemmaMsihMultiplication}.
\end{proof}

\section{\texorpdfstring{The group $\cM(n)$}{The group M(n)}}
\subsection{\texorpdfstring{The subgroup  $\cM_{0}(n)$}{The subgroup M0(n)}}\label{SectionMnM0n}

We let $\cM_0(n)$ be the subset of $\cM(n)$ containing only those matrices
that have no alternating line:
\begin{align*}
    \cM_0(n):=\{M\in\cM(n) : M = M(\sigma, w, 0),\, \sigma\in S_n\}\,,
\end{align*}
(where, as noticed from the definition~\eqref{eqDefMsih2}, the integer parameter $w$ plays no role).

The pair $(\cM_0(n),\cdot)$ is  a group with the identity
$M(\id,w,0)$. By~\eqref{eqMultiplication00A}, it follows that
$M(\sigma,w,0)^{-1}=M(\sigma^{-1},w,0)$, and the rule of multiplication is associative.
Furthermore, the transformation 
$f : S_n\to \cM_0(n)$ with
$f(\sigma):= M(\sigma,w,0)^{-1}$ is an isomorphism between~$\cM_0(n)$ and the symmetric group $S_n$.

The formulas in Theorem~\ref{LemmaMsihMultiplication} also show that
$\cM(n)$ is closed under the multiplication of matrices
and 
\begin{align*}
    M(\sigma,u,1)^{-1}= M(\sigma^{-1},\sigma(u),1).
\end{align*}
It follows that $\cM(n)$ is a group with $(n+1)!$ elements
and within, it lies the subgroup~$\cM_0(n)$ with $n!$ elements.
\begin{figure}[htb]
\centering
\hfill
\includegraphics[width=0.38\textwidth,angle = -90]{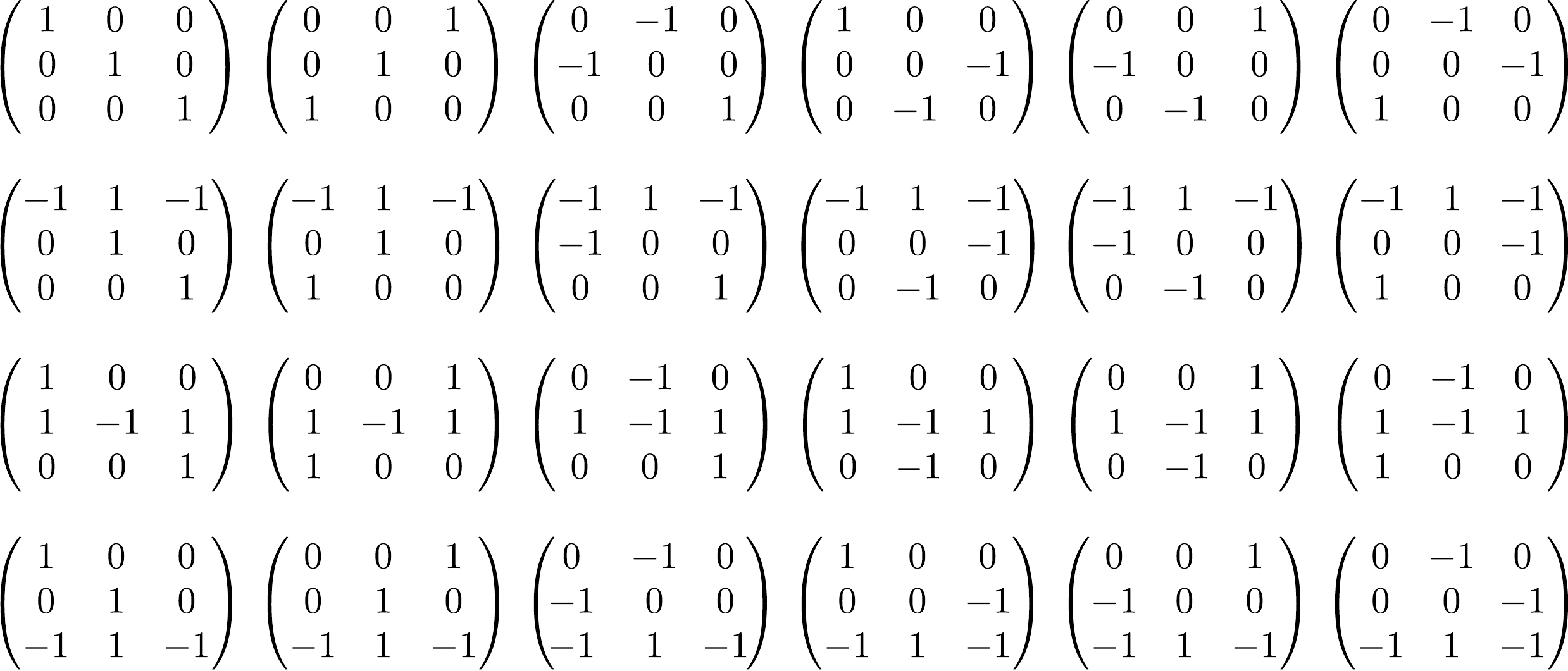}
\hfill\mbox{}
\caption{The elements of $\cM(3)$. The matrices are ordered such that the first six matrices are the elements of the subgroup $\cM_0(3)$, and on the following lines are its cosets $K_1\cM_0(3)$, $K_2\cM_0(3)$, and $K_3\cM_0(3)$.}
\label{FigureM3}
\end{figure}

The group $\cM(3)$ and the cosets $K_j\cM_0(3)$
are listed in Figure~\ref{FigureM3}.
It can be verified there that $\cM(n)$ is non-abelian 
and neither $\cM_0(n)$ is a normal subgroup of $\cM_(n)$.
This follows in general for all $n\ge 2$
by observing how the groups form a growing tower of nested groups:
\begin{align*}
     \cM(2) < \cM(3) <\cdots < \cM(n) <\cdots
\end{align*}
\begin{figure}[htb]
\centering
\hfill
\includegraphics[width=0.389\textwidth, angle = -90]{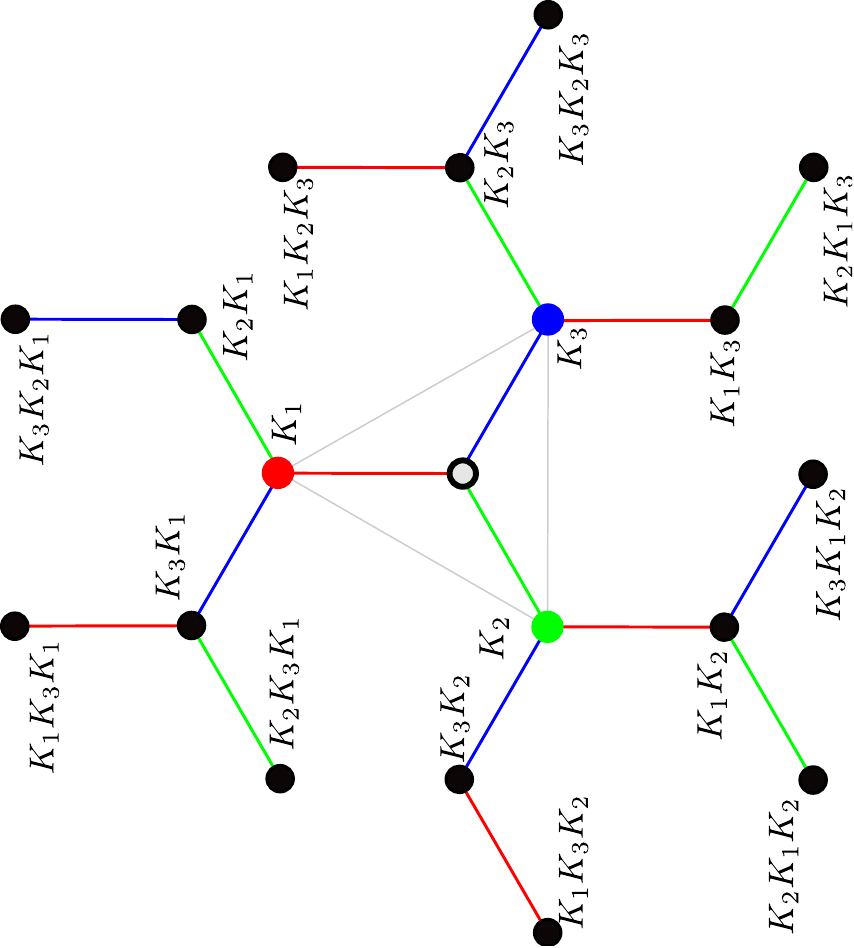}\quad
\includegraphics[width=0.389\textwidth, angle = -90]{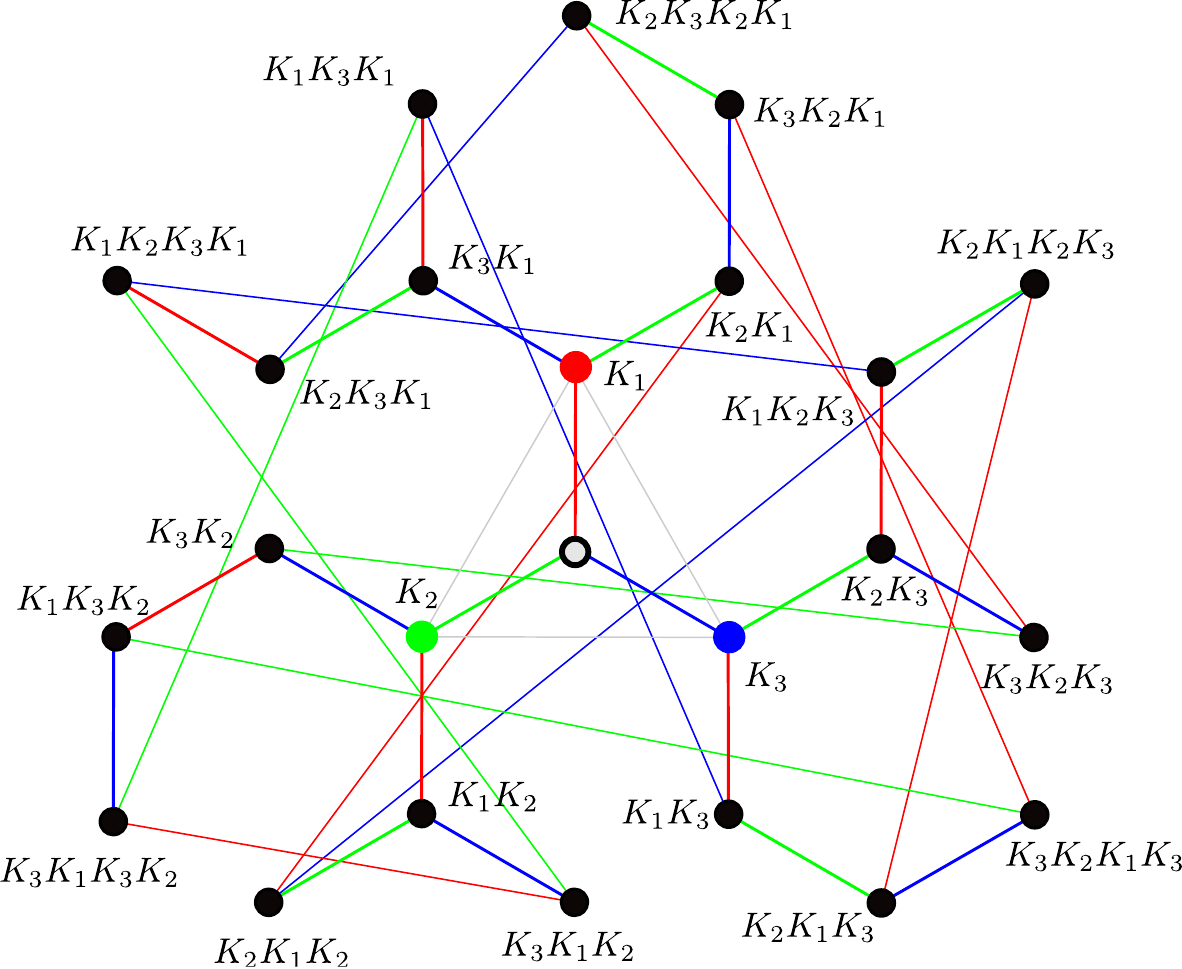}
\hfill\mbox{}
\caption{ The Cayley diagram of $\cM(3)$.
The colored arcs in red, green, blue indicate left multiplications by $K_1$, $K_2$, $K_3$, respectively. 
In the symmetric diagram on the left-side, 
only the nodes 
located at an orbit distance $d_o$ of at most $3$
from the identity are shown. 
The other five nodes at distance $4$ from the identity and the corresponding arcs are shown in the complete figure on the right-side.}
\label{FigureCayley}
\end{figure}

\subsection{\texorpdfstring{
Prior attempts to classify $\cM(n)$}{Prior attempts to classify M(n)}}\label{SectionPrior}

The representation of matrices $K_j$ as $M(\sigma,w,\epsilon)$ 
establishes an initial link between matrices and permutations.
Then, the formulas from Theorem~\ref{LemmaMsihMultiplication}, leading to the conclusion that $\cM(n)$ is a group that has $(n+1)!$ elements, just like 
$\langle K_1,\dots,K_n\rangle \subset \cM(n)$, strengthen the hypothesis that there is a closer relation with the symmetric group.

We also saw in Section~\ref{SectionMnM0n} that the subgroup $\cM_0(n)$ of matrices without alternating rows is isomorphic to $S_n$.
In particular, we noticed that if $n=3$, then $\cM(3)$ is a group with $24$ elements that has $S_3$ as a subgroup.
Since $D_{2\cdot 12}$, the dihedral group of symmetries of the regular  dodecagon,
also has $S_3$ as a subgroup, and the multiplication formulas 
in $\cM(n)$ are somewhat tangled, one might guess that $\cM(3)$ could be isomorphic to $D_{2\cdot 12}$.
However, $D_{2\cdot 12}$ has an element of order $12$, while $\cM(3)$ does not, so that $\cM(3)\not\cong D_{2\cdot 12}$.

Additionally,  by checking the sets with repetition of the orders of the elements of $\cM(3)$ and $S_4$, it can be seen that they coincide 
(apart from the identity, there are $9$ elements of order $2$, $8$ elements of order $3$, and $6$ elements of order $4$;
see Figure~\ref{FigureOrdersS4M3}).
\begin{figure}[htb]
\centering
\hfill
\includegraphics[width=0.35\textwidth, angle = -90]{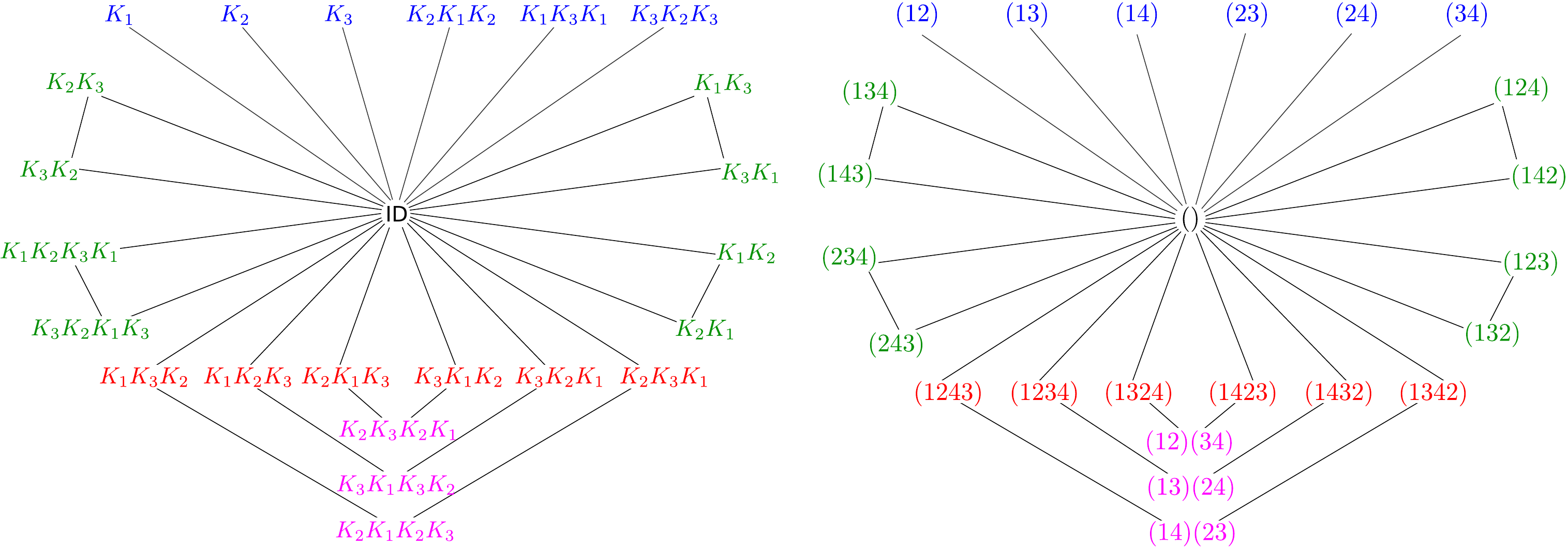}
\hfill\mbox{}
\caption{The orders of the the matrices in $\cM(3)$
and of the permutations in $S_4$.
The positions of the elements in the diagrams are corresponding
to the isomorphism~\eqref{eqIsomorphismCD}.
There are $6+3$ elements of order $2$ (the later three belonging to subgroups of order $4$), $8$ elements of order $3$, 
and $6$ elements of order $4$. 
}
\label{FigureOrdersS4M3}
\end{figure}

All of these suggest that $\cM(3)\cong S_4$, and that this isomorphism might hold in any dimension (if $n=1$ or $2$, the verification that $\cM(n)\cong S_{n+1}$ is immediate).
In Figure~\ref{FigureOrdersS4M3}, one can see the diagrams of all cyclic subgroups of $S_4$ and $\cM(3)$.
Although their structures coincide, there are numerous arrangements in which the elements of subgroups of the same order can be placed in nodes, leaving hidden the association that would actually generate an isomorphism.

Regardless, trying first with the elements of order $2$,
we should aim to find a suitable pairing between the involutions.
On the one hand, in $\cM(3)$, these are:
$K_1,K_2,K_3$,
the palindromic representatives
$K_1K_3K_1$, 
$K_2K_1K_2$, 
$K_3K_2K_3$,
and
the three elements at orbit-distance $4$ from the identity, elements that belong to the three cycles of order four:
$K_2K_1K_2K_3$, $K_3K_1K_3K_2$, and $K_2K_3K_2K_1$ (see Figures~\ref{FigureCayley} and~\ref{FigureOrdersS4M3}).
On the other hand, in $S_4$ there are
the six transpositions $(1,2);(1,3);(1,4);(2,3);(2,4);(3,4)$
and the three products of disjoint transpositions
$(1,2)(3,4)$;
$(1,3)(2,4)$; and
$(1,4)(2,3)$.

When selecting minimal generator sets for the two groups, and taking $K_1, K_2,\dots, K_n$ on the $\cM(n)$ side, the `linked' chain of transpositions
$ (1,2), (2,3), \dots, (n-1,n), (n,n+1)$ on the side of~$S_{n+1}$ 
would seem to be the natural candidate for the corresponding generators of~$S_{n+1}$. Nevertheless, this approach proves unsuccessful.
The reason stands on the fact that any two distinct operators 
$K_j$ and $K_\ell$ do not commute, 
whereas disjoint transpositions, such as $(1,2)$ and $(3,4)$, do.

Another pick that will turn out to be successful is the subject of 
the following section.

\section{\texorpdfstring{Completion of the proof that
$\cM(n)\cong S_{n+1}$
}{Completion of the proof that M(n) and S(n+1) are isomorphic}}\label{SectionIsomorphismMS}

Let us view $\cM(n)$ in a more abstract framework, that of the Coxeter groups (see~\cite{Cox1934, Cox1935, BB2005, Hum1990}).
As a finite group generated by $K_1,\dots, K_n$,
the associated matrix of the orders of the product of any two generators is
$(m_{jk})$, where  $m_{jk}=m_{kj}:=3$, if $1\le j\neq k\le n$,
according to Lemma~\ref{LemmaOrderKjl}, and $m_{jj}:=1$ if $1\le j\le n$, 
according to Lemma~\ref{LemmaInvolutionK}.
Then, the pair group-generators
$(M(n),\{K_1,\dots,K_n\})$, together with the presentation given by all relations $(K_jK_k)^{m_{jk}}=\Id$
for $1\le j,k\le n$ forms a \textit{Coxeter system},
and $M(n)$ is called a \textit{Coxeter group}.
Being a symmetric matrix that describes the relations 
between the involutions that generate the group $M(n)$,
$(m_{jk})$ is a \textit{Coxeter matrix}.

In an equivalent way, the Coxeter matrix can be represented 
as an unordered graph, called \textit{Coxeter diagram},
whose nodes are the generators $K_j$, and the edge
between any two distinct nodes $K_j$ and $K_k$ has weight $m_{jk}$ for $1\le j\neq k\le n$.
The convention is that no loop $K_jK_j$ is included in the diagram, and that the arcs are labeled by the weights only if they are strictly larger than $3$.
Hence, the Coxeter diagram of $M(n)$ is a polygon with 
$n$ vertices along with all its sides and diagonals.

Viewed as linear operators over an $n$-dimensional vector space, 
$K_j$ also verifies the conditions for being a reflection (see~\cite{Hum1990, BB2005}), 
so that $M(n)$ is also a \textit{reflexion group}.

The symmetric group $S_{n+1}$ is also a reflection group
as it identifies with a subgroup of
$O_{n+1}(\QQ):=\left\{O\in GL(n+1, \QQ) :
O^{\mathsf {T}}O=OO^{\mathsf {T}}=\Id\right\}$
of orthogonal $(n+1)\times (n+1)$ matrices,
in which permutations act on $\QQ^{n+1}$ by permuting
the subscripts of the basis vectors 
\mbox{$e_j=(0,\dots,0,1,0\dots 0)$}
(the $n+1$-tuple with all components equal to 0, 
except the~$j$th, which is $1$).
Moreover, $S_{n+1}$ is also a Coxeter group, but the presentation  
provided by the sequence of ordinary set of transpositions  
$(j,j+1)$ leads to a linear Coxeter diagram, which is fundamentally distinct from the above Coxeter diagram of $M(n)$.

In order to see that $M(n)$ is indeed isomorphic to $S_{n+1}$,
a different system of generators for~$S_{n+1}$ is needed.
One that proves suitable is the set of transpositions
$(1,j)$ with \mbox{$2\le j\le n+1$}.
These transpositions also generate $S_{n+1}$, because 
(with the multiplication of transpositions written from left to right)
$(1,j)\cdot (1,j+1)\cdot(1,j)=(j,j+1)$ shows that the ordinary transpositions $(j,j+1)$ can be obtained as products as the one from the new set.
The transpositions $(1,j)$ are also involutions and, since $(1,j)\cdot (1,k)=(1,j,k)$, it follows
$\big((1,j)\cdot (1,k)\big)^3=\id$ (the identity permutation).
Therefore, with this new set of transpositions as generators,
the associated matrix of the orders coincides with the
Coxeter matrix of $\cM(n)$. 
Thus, in this setting, $S_{n+1}$ is a Coxeter group whose
Coxeter diagram is an 
$n$-polygon together with all its edges and diagonals
un-labeled, exactly as that of $M(n)$.
Since the sets of generators of $\cM(n)$ and $S_{n+1}$ have the same cardinality, their associated Coxeter systems are 
equivalent.

In conclusion, these imply that $\cM(n)\cong S_{n+1}$
(see~\cite[Section 2.1]{Hum1990} for the general argument).
We remark that the requirement for the groups in question to be reflection groups can be bypassed, since, according 
to~\cite[Theorem 6.4]{Hum1990}, any Coxeter group is 
a finite group  if and only if it is a  reflection group.
The overall result still requires identifying how the isomorphism works, as we have seen that the choice of generator systems is essential.

For the sake of illustration, 
let $F:\{K_1,\dots,K_n\}\to \cM(n)$ be the natural embedding that maps $K_j$
to \mbox{$F(K_j):=K_j$} for all $j$. 
Then, let $\psi: \{K_1,\dots,K_n\}\to S_{n+1}$
be defined by $\psi(K_j):=(1,j+1)$.
Note that
\begin{align}\label{eqUP}
   \left(\psi(K_j) \psi(K_k)\right)^{m_{jk}}
   = \left((1, j+1)\cdot (1,k+1)\right)^{m_{jk}}
   = (1,j+1,k+1)^{m_{jk}} = \id
\end{align}
for all $j,k$ with $1\le j,k\le n$.
Then, according to the \textit{universality property}
of the groups obtained by presentations, 
there is a unique extension of $\psi$ 
to a group homomorphism 
$\Psi :\cM(n)\to S_{n+1}$
such that $\psi = \Psi\circ F $ (see~\mbox{}{\cite[Section 1.1]{BB2005}}).

Likewise, 
let $G:\{(1,2),\dots,(1,n+1)\}\to S_{n+1}$ be the natural embedding that maps $(1,j+1)$
to \mbox{$G((1,j+1))=(1,j+1)$} for all $j$. 
Then, let $\phi: \{(1,2),\dots,(1,n+1)\}\to \cM(n)$
be defined by $\phi((1,j+1)):=K_j$.
Similarly as in~\eqref{eqUP}, we  note that
\begin{align}\label{eqUPd}
   \left(\phi((1,j+1)) \phi((1,k+1))\right)^{m_{jk}}
   = \left(K_j\cdot K_k\right)^{m_{jk}}
   = \Id
\end{align}
for all $j,k$ with $1\le j,k\le n$,
according to Lemmas~\ref{LemmaInvolutionK} 
and~\ref{LemmaOrderKjl}.
Then, once we know that~\eqref{eqUPd} holds, again according to the universality property,
it follows that 
there is a unique extension of $\phi$ 
to a group homomorphism 
\mbox{$\Phi :S_{n+1} \to \cM(n)$}
such that $\phi = \Phi\circ G $.
In conclusion, we obtained the following commutative diagram
in which both homomorphisms $\Phi$ and $\Psi$ are unique:
\begin{equation}\label{eqIsomorphismCD}
  \begin{tikzcd}[
]
    \{K_1,\dots,K_n\} 
    \arrow[r,"F"]
    \arrow[swap,rd,"\psi"] 
    & \cM(n) 
    \arrow[swap, shift right=0.75ex,  d,"\Psi"] 
    & \\
     &  S_{n+1} 
         \arrow[swap,shift right=0.75ex, u,"\Phi"]
     & \{(1,2),\dots,(1,n+1)\}
     \arrow[l,"G"]
    \arrow[swap,lu,"\phi"] 
  \end{tikzcd}    
\end{equation}
Noting that both $\cM(n)$ and $S_{n+1}$ have the same cardinality $(n+1)!$
and the fact that when applied on the generators the homomorphisms
$\Psi$ and $\Phi$
satisfy
\begin{align*}
   \Phi(\Psi(K_j))&=\Phi((1,j+1))=K_j\\
\intertext{and}
    \Psi(\Phi(1,j+1))&=\Psi(K_j)=(1,j+1),
\end{align*}
it follows that $\Psi$ and $\Phi$ are the inverses of one another,
so that $\Psi$ and $\Phi$ are isomorphisms, that is,
we have proved the following theorem.
\begin{theorem}\label{TheoremIsomorphism}
$\cM(n)\cong S_{n+1}$ for any integer $n\ge 2$.
\end{theorem}

\section{\texorpdfstring{Completion of the proof of Theorem~\ref{PropositionPodDI}}{Completion of the proof of PropositionDI}}\label{SectionProofPropositionPodDI}
In order to address the inherent challenges arising from the non-commutativity of matrix multiplication, we will transfer the calculations to the symmetric group. 
Via the isomorphism~\eqref{eqIsomorphismCD}, we have
\begin{align*}
    \Psi(K_1K_{2}\cdots K_{n-1}K_n) = (1,2)(1,3)\cdots (1,n)(1,n+1).
\end{align*}
We let $\pi$ denote the product $\pi = (1,2)(1,3)\cdots (1,n)(1,n+1)$,
where we recall that the convention is that the action of the transpositions is taken in order from left to right.
Let us now compose $\pi$ with itself iteratively, starting by
checking successively the path of $1$.
Thus, we have:
\begin{align*}
    \pi(1)=2;\ \pi^2(1)=\pi(2)=3;\ 
    \pi^3(1)=\pi^2(2)=\pi(3)=4; \dots
\end{align*}
Then, by induction, it follows that
$\pi^n(1)=n+1$, which implies that $\pi^{n+1}(1)=1$,
and that no lower exponent than $n+1$ can turn $1$ into $1$.

Likewise, a circular path is also followed by any different element $k\neq 1$. It also traverses through all other values, since we have;
\begin{align*}
    \pi(k)=k+1;\ \pi^2(k)=\pi(k+1)=k+2;\dots;
    \pi^n(k)=k-1; \pi^{n+1}(k)=k, 
\end{align*}
with the convention that $n+2$ is replaced by $1$.

This completes the proof of Theorem~\ref{PropositionPodDI}.
\hfill\mbox{}$\qed$\mbox{}

\section{\texorpdfstring{Arithmetic-Geometric aspects of the operators $\K_j$}{Arithmetic-geometric aspect of the operators Kj}}

For any integer $n\ge 2$, we write $\bx=(x_1,x_2,\dots,x_n)$ and
introduce the operators \mbox{$\K_j : \ZZ^n\to \ZZ^n$} for $1\le j\le n$
defined by
\begin{align*}
    \K_j(\bx):=(x_1,\dots,x_{j-1},\br_j\bx,x_{j+1},\dots, x_n).
\end{align*}
Note that $\K_j(\bx)$ keeps fixed $n-1$ components of $\bx$, 
and the component with index $j$ is the alternate circular 
sum of all components starting with $-x_j$.
As a first consequence, it follows that $\K_j(\bx)$ is a 
\textit{homothety} with center the origin.

Using the matrices introduced by~\eqref{eqDefKj}, we have:
\begin{align}\label{eqDefOperatorKj}
    \K_j(\bx):= (K_j\bx^T)^T.
\end{align}
In this form, proving the involution from Lemma~\ref{LemmaInvolutionK} 
only requires verifying that the $x_j$ component returns to its initial value 
after two transformations applied by $\K_j$. 
Indeed, upon canceling out terms, we find that
\begin{align*}
   -\br_j\bx+x_{j+1}-\cdots+(-1)^{n+j-1}x_n+ (-1)^{j-1}x_1+\cdots+x_{j-1}=x_j.
\end{align*}

To measure the spaces between points, three natural distances 
can be used to describe geometric patterns.
These are: the \textit{ Euclidean distance} denoted  by $d_E$, 
the \textit{taxicab distance} denoted by $d_c$ which adds up the length 
of successive segments along a specified path between two points, 
and the distance $d_o$ which counts the minimum number of operators~$\K_j$ 
that must be applied to reach from one point to another. 
Note that this last distance, which is formally defined by
\begin{align*}
    d_o(\bx,\by):= \min_r\big\{
         (j_1,\dots,j_r)\in\{1,\dots,n\}^r : 
         \K_{j_1}\cdots\K_{j_r} (\bx)=\by
    \big\}
\end{align*}    
verifies the defining axioms of a distance.
We will typically use $d_o$ for points belonging to the same orbit,
so that we call $d_o$ the \textit{orbit distance}.
Likewise, we will use $d_c$ as the taxicab distance along the shortest possible path between two points within the same orbit.

\begin{figure}[htb]
\centering
\hfill
\includegraphics[width=0.26\textwidth, angle = -90]{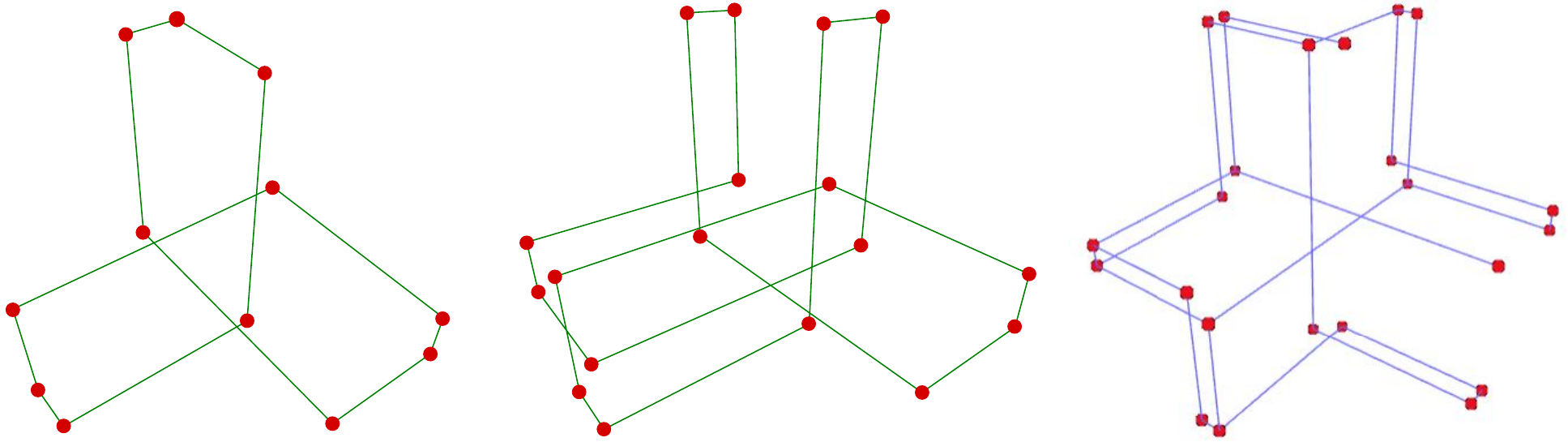} 
\hfill\mbox{}
\caption{Orbits of distinct jumps for $n=3$.
The orbit on the left-hand side has $12$ jumps and is generated by the sequence of $\K_j$'s with $j\in \{1, 2, 3, 1, 2, 3, 1, 2, 3, 1, 2, 3\}$.
The orbit on the center has $18$ jumps and is generated by the augmented sequence of~$\K_j$'s with $j\in \{2, 3, 2, 1, 3, 2, 3;\, 2, 3, 1, 2, 3, 1, 2, 3, 1, 2, 3\}$.
The orbit on the right-hand side is longer, but the nodes are not all different and the associated path is neither Eulerian nor Hamiltonian.
}
\label{Orbits12Jumps18Jumps}
\end{figure}
While starting with $\bw\in\ZZ^n$,
depending on the choice and order of the indices $j_1,j_2,\ldots$ in 
$\{1,2,\dots,n\}$, the iterated application of  
$\K_{j_1},\K_{j_2},\dots$
may lead to a path that may be closed or not (see Figure~\ref{Orbits12Jumps18Jumps}). 
Of special interest are those that are minimal, such as those 
that generate Eulerian or Hamiltonian orbits.
By jumping just randomly, the chances of finding such a path when $n$ increases are getting smaller and smaller.
To find a Hamiltonian closed path, a coordinated strategy is needed, for example, by choosing the combined one of 
walking orderly on one one-step-incomplete-orbit embedded 
in a $2$-dimensional section of $\ZZ^n$, 
followed by replacing the last step to completion with a jump to another level (an example in three dimensions is shown in Figure~\ref{Figure3dorbits}).
In the primary case $n=2$, as $K_1$ and~$K_2$ are involutions, 
there is no other choice but taking on the repetition of the two steps sequence $K_1,K_2$, which ends quickly resulting in an orbit of at most $6$ distinct edges.
Their foreground is discussed in the following.

\subsection{\texorpdfstring{The orbits if $n=2$}{The orbits if n=2}}\label{SectionAughts}
In this case, there are only two $\K$-operators:
\begin{align*}
    \K_1(x_1,x_2) = (-x_1+x_2,x_2)\ \ \text{ and }
    \K_2(x_1,x_2) = (x_1, x_1-x_2).
\end{align*}

Knowing that both $\K_1$ and $\K_2$ are involutions and according to 
Lemma~\ref{LemmaOrderKjl}, the orders of $\K_1\K_2$ and $\K_2\K_1$
are $3$, it follows that the maximum length of a cycle generated 
by any $(x_1, x_2)\in\ZZ^2$ is $6$. 
The points of the cycle are: 
\begin{equation}\label{eqOrbitPath}
  \begin{split}
    & (x_1, x_2) 
    \xleftrightarrow{\K_1} (-x_1+x_2, x_2) 
    \xleftrightarrow{\K_2} (-x_1+x_2, -x_1)
    \\
    \xleftrightarrow{\K_1} 
    & \,
    (-x_2, -x_1) 
    \xleftrightarrow{\K_2} (-x_2, x_1-x_2)
    \xleftrightarrow{\K_1} (x_1, x_1-x_2) 
    \xleftrightarrow{\K_2} (x_1, x_2).
  \end{split}
\end{equation}
and the cycle can be traversed in any order by applying the same indicated operators in reverse.
If we denote with $P_1,\dots,P_6$ the successive points in an orbit,
then we see that they group into three distinct pairs of 
diametrically opposite points at orbit distance equal to $3$:
$(P_1,P_4)$, $(P_2,P_5)$, and $(P_3,P_6)$.

Letting $P_1=(x_1,x_2)$ and applying the operators starting with $\K_1$ as 
in~\eqref{eqOrbitPath}, we have
\begin{align}\label{eqdcby3}
    d_c(P_1,P_2) = \abs{-x_1+(-x_1+x_2)} &= \boxed{\abs{2x_1-x_2}}\notag\\
         &= \abs{-(-x_1)+(x_1-x_2)} = d_c(P_4,P_5);\notag\\
    d_c(P_2,P_3) = \abs{-x_2+(-x_1)} &= \boxed{\abs{x_1+x_2}}
         = \abs{-(-x_2)+x_1} = d_c(P_5,P_6);\\
    d_c(P_3,P_4) = \abs{-(-x_1+x_2)+(-x_2)} &= \boxed{\abs{2x_2-x_1}}\notag\\
        &= \abs{-(x_1-x_2)+x_2} = d_c(P_6,P_1). \notag   
\end{align}
Adding up the lengths of the jumps, we have
$d_c(P_1,P_4) = d_c(P_1,P_2)+d_c(P_2,P_3)+d_c(P_3,P_4)$,
and similarly for the other pairs, it then follows that 
\begin{align}\label{eqCabDiameters}
    d_c(P_1,P_4) &= 
    d_c(P_2,P_5) = 
    d_c(P_3,P_6),  
\end{align}
that is, the pairs 
$(P_1,P_4)$, $(P_2,P_5)$, and $(P_3,P_6)$
that are diametrically opposite in the orbit distance $3$
are also diametrically opposite in the taxicab distance taken along the path
in the orbit $P_1,P_2,\dots,P_6$.

We let $p(\bx)$ denote the \textit{taxicab semi-perimeter} 
of the orbit that contains $\bx$, which by~\eqref{eqdcby3} 
and~\eqref{eqCabDiameters} is
\begin{align}\label{eqSemiperemeter2D}
    p(\bx) := \abs{2x_1-x_2} + \abs{x_1+x_2} + \abs{2x_2-x_1}.
\end{align}

Here are a few of the simple observations about the orbits
$o(\bx)$ generated by points~$\bx\in\ZZ$ by the iterative application of $\K_1$ and $\K_2$.
Each $\bx$ is a generator of a unique orbit.
The orbit containing the origin has only one element.
Starting from $\bx$ and applying iteratively $\K_1, \K_2,\dots$, or $\K_2, \K_1,\dots$, 
the set of nodes obtained is the same, only traversed in one direction or in the opposite direction.
The components of the nodes belonging to an orbit 
can take only six values: $\pm x_1, \pm x_2$, and $\pm(x_1-x_2)$ 
for some $x_1,x_2\in\mathbb{Z}$, each of these being taken twice.
The lattice points on the lines $y=2x$
are fixed points of $\K_1$ and 
those on $y=x/2$ are fixed points of $\K_2$.
As a result, except for the origin,
they belong to orbits with exactly three distinct points.
Half of the segments joining consecutive nodes 
of an orbit 
are horizontal and the other half are vertical.
We collect in the next proposition further properties
of the orbits in $2$-dimensions.some 
\begin{proposition}\label{PropositionOrbits2D}
Starting with $(x_1,x_2)\in\ZZ^2$, the repeated application of the operators 
$\K_1,\K_2,\K_1,\K_2,\dots$ generates cycles of length $6$.
The set of all orbits formed by either~$1$ or~$3$ points in degenerate cases,
or by~$6$ distinct points in almost all cases, 
creates a partition of~$\ZZ^2$.
The orbits have the following characteristics:
\begin{enumerate}
   \item[\namedlabel{PropositionOrbits2DA}{\normalfont{(i)}}]
    The length (or the perimeter) of each orbit is divisible by $4$.

    \item[\namedlabel{PropositionOrbits2DB}{\normalfont{(ii)}}]
     There exist orbits with perimeter congruent to $r\pmod{d}$ for any odd integer $d\ge 3$ and
    $r\in\{0,1,2,\dots,d-1\}$.

    \item[\namedlabel{PropositionOrbits2DC}{\normalfont{(iii)}}]
The bounding box of each orbit is a square
with a perimeter equal to that of the orbit.
   
\end{enumerate}
\end{proposition}
\begin{proof}

We already know from the first part of this section 
that the perimeter of an orbit~$o(\bx)$ is even and has 
the semi-perimeter $p(\bx)$ as in~\eqref{eqSemiperemeter2D}.
Then, since any integer has the same parity as its absolute 
value has, it follows
\begin{align*}
   \abs{2x_1-x_2} + \abs{x_1+x_2} + \abs{2x_2-x_1}
   &\equiv   (2x_1-x_2) + (x_1+x_2) + (2x_2-x_1)\pmod{2}\\
   &\equiv 0\pmod{2},
\end{align*}
which proves part~\ref{PropositionOrbits2DA}.

Let $d\ge 3$ be an odd integer and let 
$r\in\{0,1,2,\dots,d-1\}$.
While there are many orbits with perimeter $\equiv r\pmod{d}$
(in Figure~\ref{FigureColoredOrbits}, pick any color and 
see how many and how the points 
of the chosen color are distributed),
here we will observe that even the orbits intersecting the real axis 
have this property.
Indeed, let $\bx=(k,0)$, where $k$ is any non-zero integer. 
Then $2p(\bx)=2(\abs{2k} +\abs{k} +\abs{k})=8k$.
Now, part~\ref{PropositionOrbits2DB} follows since $d$ is odd, so that $\gcd(8,d)=1$, 
which implies that each residue class $r\pmod d$ 
is equally covered by both the set $\ZZ$ 
and the set  $8\ZZ$.

Let us now show that the smallest rectangle that encloses 
an orbit is actually a square (see three such examples 
in the graphic representation in Section~\ref{SectionAughtsGraphics}.
Let us observe that since the jumps between consecutive nodes 
of the orbit are made alternately and parallel to the two 
coordinate axes, it follows that each side of the minimal 
rectangle enclosing the orbit will have to include such a
jump (even if it is of length $0$, as it happens in 
degenerate cases).

From~\eqref{eqOrbitPath} and~\eqref{eqdcby3} we find that
there are at most three different lengths of segments 
that connect consecutive points in the orbit.
There are always three distinct segments, except in the
degenerate cases, where a jump has zero length. 
In such an instance, the orbit is formed by segments 
that have only two different lengths or,
if a node in the orbit is the origin, then all nodes 
coincide, causing all segments to be of equal length $0$.
Additionally, the six segments that make up an orbit are 
paired as one horizontal segment and one vertical segment, 
both having the same length.

The three lengths of segments of the orbit $o(\bx)$ are: 
$\abs{x_1+x_2}$, $\abs{x_1-2x_2}$, and $\abs{x_2-2x_1}$.
Thinking of the segments of $o(\bx)$ as untied at nodes,
we can then translate them in such a way that they fill the gaps on the sides of the enclosing rectangle. 
Furthermore, the pigeonhole principle implies that it cannot be otherwise than one of the three segments matches a rectangle's side
and the other two segments complement each other to equal that length.
And this happens both horizontally and vertically. 
Consequently, the bounding box of $o(\bx)$ is a 
square with side length $L(\bx):=p(\bx)/2$ and
\begin{equation}\label{eqL}
  \begin{split}
 L(\bx)    
   &=\max\big\{
   \abs{x_1+x_2},\,\abs{x_1-2x_2},\,\abs{x_2-2x_1}
   \big\}\\
   &= \frac{1}{2}\big(
    \abs{2x_1-x_2} + \abs{x_1+x_2} + \abs{2x_2-x_1}
    \big).
  \end{split}
\end{equation}
This completes the proof of part~\ref{PropositionOrbits2DC}.
\end{proof}

\subsection{\texorpdfstring{The Euclidean diameter of $o(\bx)$}
{The Euclidean diameter of o(x)}}
While in distances $d_o$ and $d_c$ all three pairs 
$(P_1,P_4)$, $(P_2,P_5)$, and $(P_3,P_6)$
of points in an orbit are diametrically opposite, 
we will see here that, 
except for degenerate cases, the Euclidean distances 
between the components of the three pairs are different.

By~\eqref{eqOrbitPath} it follows that 
\begin{align*}
    d_E(P_1,P_4) = \sqrt{2}\abs{x_1+x_2},\
    d_E(P_2,P_5) = \sqrt{2}\abs{2x_2-x_1},\
    d_E(P_3,P_6) = \sqrt{2}\abs{2x_1-x_2},
\end{align*}
so that the Euclidean diameter 
of the orbit $o(\bx)$ is
\begin{align*}
    \diam_E(o(\bx)) = \sqrt{2}\max\big\{
        \abs{x_1+x_2},\,
        \abs{2x_2-x_1},\,
        \abs{2x_1-x_2}
    \big\}.
\end{align*}

\section{Representatives of the orbits}\label{SectionRepresentativeOrbits}\label{SectionRepresentatives}
Let us observe here how applying the mappings $\K_1$ and $\K_2$ results in transformations of well-determined domains. These transformations are essentially symmetries consisting of reflections and vertical or horizontal translations.

First, we note that the structure of the transformations is marked 
by certain directing lines. These are:
the line $y = 2x$ fixed by $\K_1$, and the line $y = x/2$ fixed by $\K_2$; the first 
diagonal that maps into the $y$-axis through $\K_1$ and respectively into the 
$x$-axis through $\K_2$. In particular, the origin remains fixed by both 
$\K_1$ and $\K_2$.
Following this, we can see that horizontal segments of equal length 
symmetrically positioned with respect to a point on the line $y=2x$ are mirrored 
in each other through $\K_1$. 
The same applies to vertical segments of equal length symmetrically positioned 
with respect to a point on the line $y=x/2$ are mirrored in each other 
through $\K_1$. 

Thus, let $M>0$ and the points be denoted as in 
Figures~\ref{FigureMirrorsTOP} and~\ref{FigureMirrorsRIGHT}, 
namely $A = (M,M)$, $B=(0,M)$, $C=(-M,0)$, and so on, 
in counterclockwise order, till $F=(M,0)$ and $A$, again; 
and the corresponding midpoints denoted 
by $\alpha = (M/2,M)$, $\beta = (-M/2,M/2)$, $\gamma = (-M,-M/2)$, and so on $\eta = (M,M/2)$.
\begin{figure}[t]
\centering
\hfill
\includegraphics[width=0.32\textwidth]{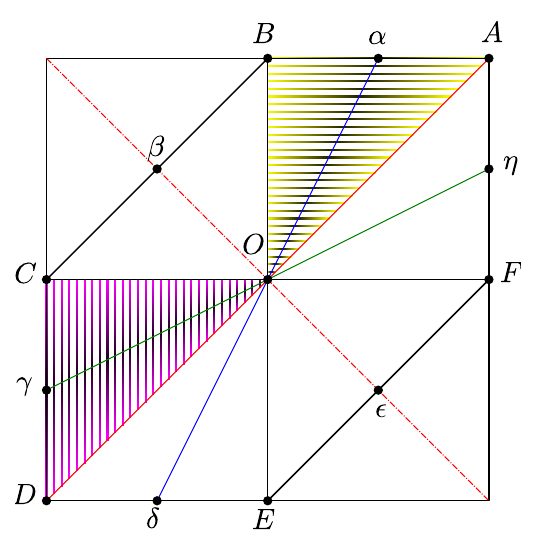}
\includegraphics[width=0.32\textwidth]{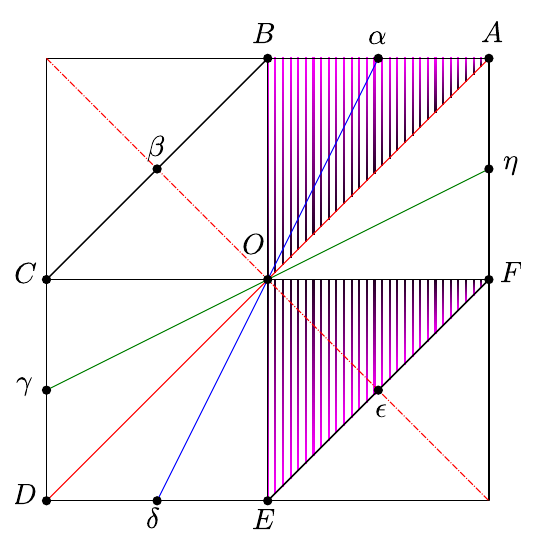}
\includegraphics[width=0.32\textwidth]{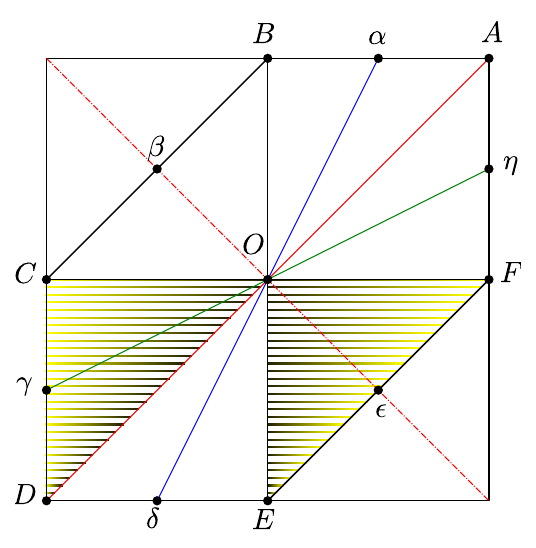}
\hfill\mbox{}
\caption{
Reflections and transformations through horizontal or vertical sliding of the top-right triangle $\triangle A\alpha O$ by the successive application of $\K_1$ and $\K_2$.
The corresponding points on the paired segments are indicated by the same color intensity.
}
\label{FigureMirrorsTOP}
\end{figure}
There are two cases to consider based on a starting segment, whether it is horizontal or vertical. 
Successive applications of the involutions $\K_1$ and $\K_2$ transform a segment back and forth, by maintaining the sequential order of points in segments from one end to the other end.

Thus, starting with the top-right segment $Aa$, the transformations are:
\begin{equation}\label{eqSegmentsPathTOP}
  \begin{split}
    & A\alpha 
    \xleftrightarrow{\K_1} B\alpha 
    \xleftrightarrow{\K_2} E\epsilon
    \xleftrightarrow{\K_1} 
    D\gamma 
    \xleftrightarrow{\K_2} C\gamma
    \xleftrightarrow{\K_1} F\epsilon 
    \xleftrightarrow{\K_2} A\alpha,
  \end{split}
\end{equation}
while starting with the right-up segment $A\eta$, produces the cycle:
\begin{equation}\label{eqSegmentsPathRIGHT}
  \begin{split}
    & A\eta
    \xleftrightarrow{\K_1} B\beta
    \xleftrightarrow{\K_2} E\delta
    \xleftrightarrow{\K_1} 
   D\delta
    \xleftrightarrow{\K_2} C\beta
    \xleftrightarrow{\K_1} F\eta 
    \xleftrightarrow{\K_2} A\eta,
  \end{split}
\end{equation}

\begin{figure}[htb]
\centering
\hfill
\includegraphics[width=0.32\textwidth]{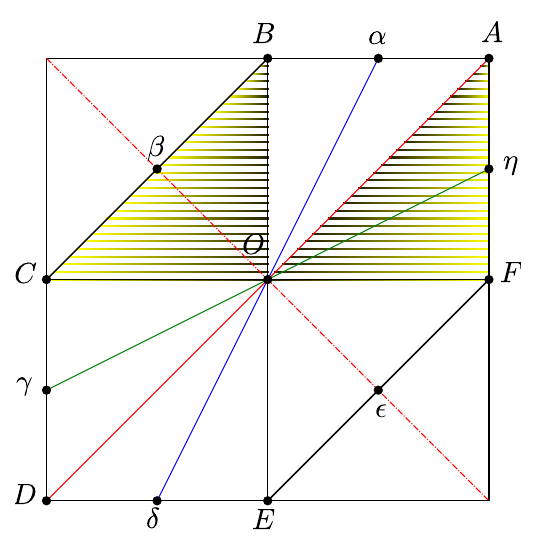}
\includegraphics[width=0.32\textwidth]{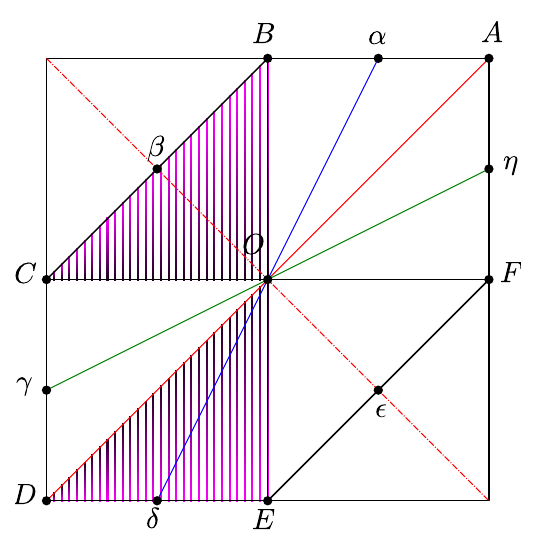}
\includegraphics[width=0.32\textwidth]{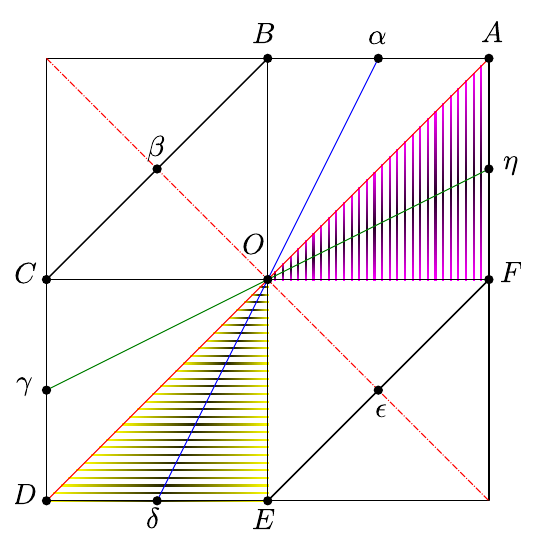}
\hfill\mbox{}
\caption{Reflections and transformations by horizontal or vertical sliding of the right-up triangle $\triangle A\eta O$ through the successive application of $\K_1$ and $\K_2$.
Points on the matched segments are highlighted with the same color intensity.
}
\label{FigureMirrorsRIGHT}
\end{figure}

Letting $M$ vary and taking the union of segments $A\alpha$, 
we find that the transformations of~$\triangle A\alpha O$ are:
\begin{equation}\label{eqTrianglesPathTOP}
  \begin{split}
    & \triangle A\alpha O
    \xleftrightarrow{\K_1} \triangle B\alpha O 
    \xleftrightarrow{\K_2} \triangle E\epsilon 0
    \xleftrightarrow{\K_1} 
    \triangle D\gamma O
    \xleftrightarrow{\K_2} \triangle C\gamma O
    \xleftrightarrow{\K_1} \triangle F\epsilon O 
    \xleftrightarrow{\K_2} \triangle A\alpha O.
  \end{split}
\end{equation}
Likewise, the cycle of transformations of $\triangle A\eta O$ is:
\begin{equation}\label{eqTrianglesPathRIGHT}
  \begin{split}
    & \triangle A\eta O
    \xleftrightarrow{\K_1} \triangle B\beta O
    \xleftrightarrow{\K_2} \triangle E\delta O
    \xleftrightarrow{\K_1} 
   \triangle D\delta  O
    \xleftrightarrow{\K_2} \triangle C\beta  O
    \xleftrightarrow{\K_1} \triangle F\eta  O
    \xleftrightarrow{\K_2} \triangle A\eta  O.
  \end{split}
\end{equation}
We denote $\cH:=ABCDEF$,
and we let $\cH'$ denote the remaining part of the hexagon $\cH$ from which we remove segments $\alpha\delta$ 
and $\gamma\eta$, 
as well as the axes $BE$, and $CF$.
Thus, $\cH'$ decomposes into two sets of six triangles each, 
as outlined in sequences~\eqref{eqTrianglesPathTOP} and~\eqref{eqTrianglesPathRIGHT}.
Then, by choosing any triangle from~\eqref{eqTrianglesPathTOP}
together with any triangle from~\eqref{eqTrianglesPathRIGHT}, we obtain a set of generators for all orbits that have points in $\cH'$.
It is worth noting the self-reflection of triangles 
$\triangle BOC$ and $\triangle EOF$ with respect to the medians along the secondary diagonal, transformations that occur after the successive application of three basic operators, that~is,
\begin{equation}\label{eqTrianglesPathNQSE}
  \begin{split}
    & \triangle OB\beta
    \xleftrightarrow{\K_2\K_1\K_2} \triangle OC\beta
\ \ \ \text{and}\ \ \
\triangle OE\epsilon 
    \xleftrightarrow{\K_1\K_2\K_1} \triangle OF\epsilon.
  \end{split}
\end{equation}
As the distance $d_o$ remains constant between any two 
corresponding points of the two pairs of triangles, 
we can express this shortly as the fact that the distance $d_o$ between the pairs of triangles is $3$:
\begin{align*}
    d_o(\triangle OB\beta, \triangle OC\beta)=3
    =d_o(\triangle OE\epsilon, \triangle OF\epsilon).
\end{align*}
This is in contrast with the mirrored triangles in quadrants $\mathrm{Q}_1$ and
$\mathrm{Q}_2$, those that contain the set of representatives that generate all the orbits, and which are actually neighbors in distance $d_o$:
\begin{align*}
    d_o(\triangle OA\alpha, \triangle OB\alpha)=1
    =d_o(\triangle OA\eta, \triangle OF\eta).
\end{align*}

\section{Arithmetic properties of orbits in 2D}\label{SectionArithmeticOn2Dorbits}

\subsection{The projections of the orbits onto the unit circle.
Proof of Theorem~\ref{TheoremProjection2D}}\label{SectionProjections2D}
We group points in orbits into two categories: 
those that are diametrically opposite in the Euclidean distance and those that are not. 
We color the points thus grouped in two colors and project them onto 
the unit circle. 
As the number of orbits generated by all points in a box becomes 
larger and larger, we observe that the projections cluster and are dense
in four distinct arcs in which the circle is partitioned. 
In Figure~\ref{FigureEuclideanDiametersProjections}, patterns of colored nodes are shown 
based on their belonging or not to Euclidean diameters of their orbits.
Also, one can see there an example of the clustering of node projections on the unit circle.
Our goal in this section is to determine the length of these arcs in the limit.

The first thing to notice is that from the symmetry of the operators $\K_j$, which implies the symmetry of the points in the orbits, whose coordinates are given in~\eqref{eqOrbitPath}, 
it follows that the four arcs of the projections are symmetric in pairs with respect 
to the origin
(see Figure~\ref{FigureEuclideanDiametersProjections}), 
and therefore, it is sufficient to find the length of just one of them.

Counting the length of the paths in an orbit along the lines of coordinates,
that is, measuring with the taxicab distance, we see that each point is in this sense 
diametrically opposite to the one obtained by applying successively 
three operators $\K_j$ of which any two consecutive are distinct.
On the other hand, when the points do not coincide geometrically, 
one checks with the formulas in~\eqref{eqOrbitPath} that in only one of these pairs
the points are diametrically opposite in the Euclidean distance.
To see this explicitly, if $o(\bx)$ is the orbit containing the point $\bx=(x_1,x_2)$,
the formula for the Euclidean diameter is then: 
\begin{align}\label{eqEuclideanDiameter}
    \diam_E(o(\bx)) = \sqrt{2}\max\big\{
    \abs{x_1+x_2},\, \abs{2x_1-x_2},\, \abs{2x_2-x_1} \big\}.
\end{align}

After examining all pairs of nodes that are part of the same orbit using sequences~\eqref{eqSegmentsPathTOP}
and~\eqref{eqSegmentsPathRIGHT}, we derive that 
in the Euclidean distance 
the only nodes that belong to the diameter of an orbit
are either in $\triangle A\alpha O$, with 
the diametrically opposite point in $\triangle D\delta O$,
or in~$\triangle A \eta O$, with the diametrically opposite point in $\triangle D \gamma O$.
(All these pairs of diametrically opposite points in
the Euclidean distance are symmetric with respect to the origin.)

Now these can answer the question of a viewer placed at the origin, who, looking around, sees the lattice points colored in two colors, depending on whether they are or are not part of a diameter of a certain orbit.
Thus, an observer viewing the lattice points as projected onto the unit circle as points
$P=(\cos(2\pi\theta, \sin(2\pi\theta))$, with $\theta\in\QQ$,
will see them as originating from diametrically opposite 
points on arcs determined by the two smallest angles between 
the lines $y=x/2$ and $y=2x$, 
while on the remaining two arcs, 
as originating from orbital points that are closer to
each other than those diametrically opposite belonging to the same orbit.
In this sense, it can be said that looking around to all lattice points in~$\ZZ^2$, 
the observer from $O$ sees a point that belongs to the 
Euclidean diameter of an orbit with a probability of 
$1/4$\footnote{If $R>0$ is sufficiently large, in the square $[-R,R]^2\cap\ZZ^2$,
the regions containing the diametrically opposite points have area $2\cdot R^2/2+O(R)$,
while the total area is $4R^2$.
}
(see Figure~\ref{FigureEuclideanDiametersProjections}).
\medskip

To address the probabilistic aspects raised in 
Theorem~\ref{TheoremProjection2D},
we consider the sets of lattice points:
\begin{align*}
   \cD_E(M) &:= \left\{
    \bx\in [0,M]^2\cap\ZZ^2 :
        \begin{array}{l}
        \text{there exists $\bx'\in o(\bx)$ such that} \\
        \diam_E(o(\bx)) = d_E(\bx,\bx')
        \end{array}
        \right\}\\
\shortintertext{and}
   \overline{\cD}_E(M) &:= \left\{
    \bx\in \cH\cap\ZZ^2 : 
            \begin{array}{l}
        \text{there exists $\bx'\in o(\bx)$ such that} \\
        \diam_E(o(\bx)) = d_E(\bx,\bx')
        \end{array}
        \right\}.
\end{align*}
Then, since $\triangle A\alpha O \cup\triangle A\eta O$
can be taken as a set of representatives of the generators of any orbit that has nodes in $\cH'$, it follows that
\begin{align}\label{eqProbabilityDE}
    \# \cD_E(M) &= \area(\triangle A\alpha O)
    +\area(\triangle A\eta O) +O(M) 
    = \sdfrac{1}{2}M^2 +O(M)\\
\shortintertext{and}
    \# \overline{\cD}_E(M) &= 
    M^2 +O(M).
\end{align}
These can be reformulated by saying that the proportion
of points $\bx\in[0,M]^2\cap\ZZ^2$ that belong to the
Euclidean diameter of $o(\bx)$ is $1/2 +O(M^{-1})$
and the proportion of points $\bx\in \cH\cap\ZZ^2$
that belong to the
Euclidean diameter of $o(\bx)$ is $1/3 +O(M^{-1})$.

\subsection{Average length of diameters and perimeters of orbits}
Considering the representatives of the orbits obtained from 
the analysis in Section~\ref{SectionRepresentatives}, 
we can calculate the average length of diameters and 
perimeters of the orbits with nodes in specific sets.
As an example, we will consider here the orbits generated by points $\bx\in  [0,M]^2\cap\ZZ^2$.

We have seen in Section~\ref{SectionRepresentatives}, 
that a good set of representatives that generates distinct orbits 
and additionally contains the points belonging to their Euclidean diameters is 
the union of the two triangles adjacent to the diagonal $y=x$.
Furthermore, since the points 
$\bx \in \triangle A\alpha O \cup\triangle A\eta O$ uniquely identify all orbits with nodes in 
$[0,M]^2\cap\ZZ^2$, the averages we will determine will have the same values even
in the version where they would be defined for all points in
the hexagon $\cH=ABCDEF$, instead of those in the square $[0,M]^2\cap\ZZ^2$.

For any positive integer $M$, we let 
$N_o(M)$ denote the number of distinct orbits with a
node in the square of side $M$ and a vertex at $O$ and by 
$\Av_{\diam_E}(M)$ the average length of the Euclidean diameter
of these orbits. Thus,
\begin{align*}
   N_o(M) := \#
   \big\{o(\bx) : \bx \in [0,M]^2\cap\ZZ^2\big\}   
\end{align*}
and
\begin{align*}
   \Av_{\diam_E}(M) := \frac{1}{N_o(M)}
   \sum_{\bx\in \triangle A\alpha O \cup\triangle A\eta O}
   \diam_E\big(o(\bx)).
\end{align*}
In the above range of summation for $\Av_{\diam_E}(M)$,
the Euclidean diameter in~\eqref{eqEuclideanDiameter}
reduces to just $\sqrt{2}(x_1+x_2)$, and because of the symmetry, 
the same result is obtained if we just average on one of the triangles. 
By choosing the lower one, for the sake of simplicity, it follows that
\begin{equation}\label{eqAverageEuclideanDiameter}
 \begin{split}
   \Av_{\diam_E}(M) &= \sdfrac{4\sqrt{2}}{M^2}
   \sum_{\bx\in \triangle A\eta O} (x_1+x_2) + O(1)\\
   &= \sdfrac{4\sqrt{2}}{M^2}
   \sum_{0\le x_1 \le M} \sum_{\frac{x_1}{2}<x_2< x_1} (x_1+x_2) + O(1)\\
   &= \sdfrac{4\sqrt{2}}{M^2}
   \sum_{x_1=0}^M \sdfrac{7}{8}x_1^2 + O(1)\\
    &= \sdfrac{7\sqrt{2}}{6} M + O(1).   
  \end{split}
\end{equation}

Then, \eqref{eqAverageEuclideanDiameter}
implies that 
the averages on $[0,M]^2\cap\ZZ^2$,
as well as on $\cH$, 
of the size of the edge of the bounding box of an orbit
$\Av_L(M)$, and of the perimeter of an orbit
$\Av_{2p}(M)$ are:
\begin{align}\label{eqAverageSizebbPerimeter}
   \Av_{L}(M) = \sdfrac{7}{6} M + O(1)
   \ \ \text{ and }\ \ 
   \Av_{2p}(M) = \sdfrac{14}{3} M + O(1).
\end{align}

\subsection{The number of orbits of a given length}\label{SectionAughtsGraphics}
Given a fixed integer $X>0$, we want to count the number
$N_{o(\bx),X}$
of distinct orbits $o(\bx)$ with perimeter $X$.
From Proposition~\ref{PropositionOrbits2D}, we already know that 
$N_{o(\bx),X}=0$, if $4\not\mid X$.
For $X$ that is divisible by $4$, it is enough to check only on the generators $\bx$ that are 
in the first quadrant between the lines
$y=x/2$ and $y=2x$. Thus, we have
\begin{align*}
  N_{o(\bx),X}&:=\#\{ o(\bx) : \bx\in\ZZ^2, \ 2p(\bx)=X \}\\
  &\; = \#\{\bx\in\ZZ^2 : 
  x_1/2\le x_2\le 2x_1,\ 4(x_1+x_2)=X\}\\
  &\; = \#\{\bx\in\ZZ^2 : 
  3x_1/2\le X/4\le 3x_1\},
\end{align*}
because the perimeter of an orbit
generated by a point within the specified domain
equals~$4$ times the size of the edge of its bounding box.
The condition on the last line above is the same as $X/12\le x_1\le X/6$, hence we can conclude that:
\begin{align}\label{eqNumberOfOrbitsPerimeterX}
  N_{o(\bx),X} = 
  \left\lfloor\frac{X}{6}\right\rfloor
  -\left\lceil\frac{X}{12}\right\rceil +1, \ \ \text{ if } 4\mid X.
\end{align}
For example, if $X=100$, in the whole $\ZZ^2$, 
there are exactly
$N_{o(\bx),100} = 
  \left\lfloor\frac{50}{3}\right\rfloor
  -\left\lceil\frac{50}{6}\right\rceil +1
  = 16-9+1 = 8$ orbits with perimeter $100$
(three of them are shown in Figure~\ref{FigureOrbitsP100}),
while there are $9$ distinct orbits with perimeter~$96$.
\begin{figure}[htb]
\centering
\hfill
\includegraphics[width=0.3\textwidth, angle=-90]{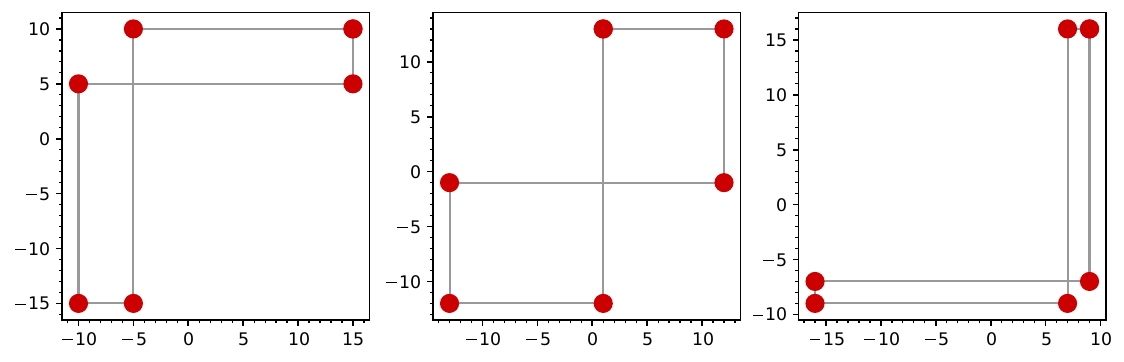}
\hfill\mbox{}
\caption{Three orbits of length $100$.}
\label{FigureOrbitsP100}
\end{figure}

The exact formula~\eqref{eqNumberOfOrbitsPerimeterX}
can be used to find a different way of computing 
the mean Euclidean diameters of the orbits. We do that by calculating the average
on the perimeters.
First, if $T\ge 0$, then the number of orbits that have perimeter at most $T$ is:
\begin{align*}
  N_{o(\bx),P}(T) 
  &\; = \#\{\bx\in\ZZ^2 : 2p(o(\bx)) \le T\}
   = \sum_{\substack{\bx\in\ZZ^2\\ 0\le 2p(o(\bx)) \le T}} 1.
\end{align*}
Making use of formula~\eqref{eqNumberOfOrbitsPerimeterX} for $X$ divisible by $4$
and knowing that there are no orbits with perimeter $\not\equiv 0\pmod 4$,
we have
\begin{equation}\label{eqNumberOfALLOrbitsPerimeterT}
  \begin{split}
  N_{o(\bx),P}(T) 
  & = \sum_{\substack{0\le X \le T\\ 4\mid X}} 
  \left(
 \left\lfloor\frac{X}{6}\right\rfloor
  -\left\lceil\frac{X}{12}\right\rceil +1\right)\\
  &= \sum_{0\le Y \le T/4} 
  \left(\frac Y3  +O(1)\right) = 
  \frac{1}{96}T^2 +O(T).
  \end{split}
\end{equation}
 Secondly, the total sum of the perimeters of the orbits in question is:
\begin{equation}\label{eqTotalPerimetersT}
  \begin{split}
  S_{o(\bx),P}(T) 
  & = \sum_{\substack{0\le X \le T\\ 4\mid X}} 
  X\left(
 \left\lfloor\frac{X}{6}\right\rfloor
  -\left\lceil\frac{X}{12}\right\rceil +1\right)\\
  & = \sum_{0\le Y \le T/4} \frac 43 Y
  \big( Y +O(1)\big) = 
  \frac{1}{144}T^3 +O(T^2).
  \end{split}
\end{equation}
On combining~\eqref{eqNumberOfALLOrbitsPerimeterT} and~\eqref{eqTotalPerimetersT},
it yields the average size of a perimeter of the orbits in the range~$[0,T]$:
\begin{equation}\label{eqAveragePerimetersT}
  \begin{split}
  \Av_{o(\bx),P}(T) 
  & := \frac{S_{o(\bx),P}(T)}{N_{o(\bx),P}(T)} 
  = \frac{2}{3}T +O(1).
  \end{split}
\end{equation}

\subsection{\texorpdfstring{Another average of the perimeters of orbits.
Proof of Theorem~\ref{TheoremPerimeterOrbitsDisk}}{Another average of the perimeters of orbits.
Proof of Theorem PerimeterOrbitsDisk}}\label{SectionSqrt5}

For $R>0$, we consider the Euclidean ball
\begin{align*}
 \cB_{E}(R) := \{\bx\in\ZZ^2 : \Vert\bx\Vert_E\le R\}
\end{align*}
and the average of the lengths of the orbits with a node in $\cB_{E}(R)$ defined by
\begin{align}\label{eqAvPDisk}
 \Av_{\cB_{E},P}(R) = \frac{1}{\#\cB_{E}(R)}
 \sum_{\bx\in \cB_{E}(R)}\length(o(\bx)).
\end{align}
Unlike the types of averages we have previously addressed, 
here the orbits are counted multiple times, but not always 
by the same number of times, a phenomenon that is caused
by the mismatch between the set of representative generators of the orbits 
and lattice points in the Euclidean ball $\cB_{E}(R)$.
Based on the conclusions from Section~\ref{SectionRepresentativeOrbits} and formula~\eqref{eqSemiperemeter2D},
we will split~$\cB_{E}(R)$ into twelve domains where the perimeter of the orbits passing through them is calculated using a corresponding associated formula.
Essentially, these can be narrowed down to 
just three areas
where the perimeter of the orbits is obtained as follows:
\begin{align*}
 \mathrm{length}(o(\bx)) = 
 \begin{cases}
 4(x_1+x_2), & \text{ if } 0\le x_1\le x_2\le 2x_1;\\
 4(2x_2-x_1), &  \text{ if } 0\le 2x_1\le x_2;\\
 4(x_1-2x_2), &  \text{ if } x_2 \le -x_1\le 0.
 \end{cases}
\end{align*}
Correspondingly, we define the following circular sectors:
\begin{alignat*}{2}
 \triangle_1&:= \{\bx\in [-R,R]^2\cap\ZZ^2 : 0\le x_1\le x_2\le 2x_1\}, \quad 
 &\cS_1 :=\triangle_1\cap \cB_{E}(R),\\
 \triangle_2 &:= \{\bx\in [-R,R]^2\cap\ZZ^2 : 0\le 2x_1\le x_2 \}, \quad 
 &\cS_2 :=\triangle_2\cap \cB_{E}(R),\\
 \triangle_3 &:= \{\bx\in [-R,R]^2\cap\ZZ^2 : x_2 \le -x_1\le 0\}, \quad 
 &\cS_3 :=\triangle_3\cap \cB_{E}(R).
\end{alignat*}

Besides $\cS_1$, there are three other sectors in which analogous points 
are nodes of orbits with the same perimeters.
These are sector $\cS_1'$, the reflection of $\cS$ across the diagonal $y=x$, 
and the symmetrical sectors of $\cS_1$ and $\cS_1'$ with respect to the origin,
triangles located in quadrant $\mathrm{Q}_3$.
The same occurs with sector $\cS_2$.
That is, the analogous points in the triangle $\cS_2'$ 
reflected over the diagonal $y=x$ and in the triangles 
in quadrant $\mathrm{Q}_3$, which are symmetrical 
to $\cS_2$ and $\cS_2'$ with respect to the origin, 
belong to orbits with the same perimeter.
And furthermore, the analogous points in $\cS_3$, those
in its reflection $\cS_3'$ across the diagonal $y=-x$,
as well as the points in their symmetrical images in quadrant 
$\mathrm{Q}_2$ with respect to the origin,
are nodes of orbits with the same perimeter.

We let $\phi = \arctan 2$ denote the angle of the half line between $\triangle_1$ and $\triangle_2$.
Then, taking into account the representatives described above, it follows:
\begin{align*}
  \sum_{\mathbf{x}\in \mathcal{M}_{E}(R)}\mathrm{length}(o(\mathbf{x}))
 &=
 16\sum_{\bx\in\cS_1}(x_1+x_2) 
 +16 \sum_{\bx\in\cS_2}(2x_2-x_1)
 +16 \sum_{\bx\in\cS_3}(x_1-2x_2)
 +O(R^2).
\end{align*}
In polar coordinates, this equals
\begin{align*}
 \sum_{\mathbf{x}\in \mathcal{M}_{E}(R)}\mathrm{length}(o(\mathbf{x}))
 &=
 16\iint\limits_{[0,R]\times[\pi/4,\phi] }(\cos \theta + \sin\theta) \rho^2 d\rho d\theta\\
 &\quad +
 16\iint\limits_{[0,R]\times[\phi,\pi/2] }(2\sin\theta - \cos \theta ) \rho^2 d\rho d\theta\\
  &\quad +
 16\iint\limits_{[0,R]\times[3\pi/2,7\pi/4] }(\cos\theta - 2\sin \theta ) \rho^2 d\rho d\theta + O(R^2), 
\end{align*}
which in the end reduces to
\begin{align*}
 \sum_{\mathbf{x}\in \mathcal{M}_{E}(R)}\mathrm{length}(o(\mathbf{x}))
 &= \frac{16}{3}R^3  \big(\sqrt{2}/2+2\sin\phi+\cos\phi
\big)+ O(R^2). 
\end{align*}
On inserting this result into~\eqref{eqAvPDisk}, we obtain
\begin{align*}
 \Av_{\cB_{E},P}(R) =  \frac{8}{3\pi}\big(\sqrt{2} + 2\sqrt{5}\big)R +O(1),
\end{align*}
which concludes the proof of Theorem~\ref{TheoremPerimeterOrbitsDisk}.

\subsection{\texorpdfstring{Orbits with length $\equiv r\pmod d$.
Proof of Theorem~\ref{Theorem_rmodd}}{Orbits with length = r pmod d. Proof of Theorem rmodd}}

Given an integer $d\ge 2$ and a class of representatives~$r$, 
with $0\le r\le d-1$,
we color the lattice points $\bx\in\ZZ^2$ 
by the color $r$, chosen distinct for each $r$,
if the length of $o(\bx)$ is $\equiv r \pmod d$.
Our aim here is to estimate the number of lattice points of color $r$.

According to part~\ref{PropositionOrbits2DA} of Proposition~\ref{PropositionOrbits2D}, the length of an orbit is  always a multiple of $4$, therefore the result will be different  
in cases where $d$ is odd, $d$ is divisible by~$2$ but not by~$4$, 
or if~$d$ is divisible by~$4$.

For any integer $M>0$, 
on counting the orbits based on $[0,M]^2\cap\ZZ^2$,
we let $N_{r,d}(M)$ denote the number of orbits whose nodes 
are colored by color $r$, that is,
\begin{align*}
    N_{r,d}(M) :=\#\{ o(\bx) : \bx\in[0,M]^2\cap\ZZ^2, \ 2p(\bx)\equiv r\pmod d \}.
\end{align*}

\vspace{1mm}
($i$) Case $d$ is odd.
The description of the representatives of the orbits in Section~\ref{SectionRepresentativeOrbits}, for each $1\le m\le M$
and $\bx\in [m/2,m]\times\{x_2\}$, 
implies that the orbit $o(\bx)$ has length 
$2p\big(o(\bx)\big) = 4(x_1+x_2)$. Then
\begin{align*}
   2p\big(o(\bx)\big)  \equiv r\pmod d \ \  \text{ if and only if }\ \ 
   x_1 \equiv 4^{-1}s\pmod d,
\end{align*}
where $4^{-1}$ is the inverse of $4$ modulo $d$ and
$s:=r-4x_2\in\{0,1,2,\dots,, d-1\}$
runs on the same set of representatives, as well as $r$ does.
Since the number of solutions of the later congruence,
as $m/2\le x_1\le m$, is $(1/d)\cdot(m/2) + O(1)$ for any 
$s$, it follows that
\begin{align}\label{eqcongruentdodd}
     N_{r,d}(M) = 2\sum_{0\le m \le M}
     \left(\frac{1}{2d}m +O(1)\right) 
     = \frac{1}{2d}M^2 + O(M)\ \text{ for } r\in\{0,1,\dots,d-1\}.
\end{align}
(Here, the coefficient $2$ accounts for the other part of equal size
of the generating set of the orbits, where $x_1\ge x_2$.)
Therefore, if $d$ is odd, the nodes of the orbits are colored in~$d$ different
colors, and the colors are roughly evenly distributed among the lattice points.

\vspace{1mm}
($ii$) Case $2 \edv d$. Let $d=2d'$, so that $2\not\mid d'$.
Then, with the same notation as in the previous case, 
$4x_1\equiv s\pmod {2d'}$ has no solution if~$s$ is odd
and has the same number of solutions as 
the congruence  $x_1\equiv 2^{-1}s'\pmod {d'}$, where 
$2^{-1}$ is the inverse of $2$ modulo $d'$ and $s':=s/2$.
This equals  $(1/d')\cdot(m/2) + O(1)$ for any 
$s$, as $m/2\le x_1\le m$.
Then, noting that~$r$ and~$s$ have the same parity, we derive
\begin{align}\label{eqcongruent2edvd}
     N_{r,d}(M) = 2\sum_{0\le m \le M}
     \left(\frac{1}{2d'}m +O(1)\right) 
     = \frac{1}{d}M^2 + O(M)\ \text{ for } r \text{ even, } 0\le r\le d-1.
\end{align}
Here again, the coefficient $2$ accounts for the generators of the orbits that are under the diagonal $y=x$.
In particular,~\eqref{eqcongruent2edvd} shows that if $d=2$, then all nodes of all orbits are colored the same, and in general,
the nodes are colored in $d/2$ different colors, uniformly distributed across the plane.

\vspace{1mm}
($iii$) Case $4 \mid d$. Let $d=4d''$.
Then, with the same notation as before, 
\mbox{$4x_1\equiv s\pmod {4d''}$} has no solution if~$s\not\equiv 0\pmod 4$
and has the same number of solutions as 
the congruence  \mbox{$x_1\equiv s''\pmod {d''}$}, where $s'':=s/4$.
The latter congruence has \mbox{$(1/d'')\cdot(m/2)+O(1)$} solutions for
$x_1\in [m/2,m]$. Then, since $r\equiv s\pmod 4$, it follows that
\begin{equation}\label{eqcongruent4edvd}
  \begin{split}
     N_{r,d}(M) &= 2\sum_{0 \le m \le M}
     \left(\frac{1}{2d''}m +O(1)\right)\\ 
     &= \frac{2}{d}M^2 + O(M)\ \text{ for } 
          r \equiv 0\pmod 4,\  0\le r\le d-1.
  \end{split}          
\end{equation}
Thus, if $4\mid d$, each orbit is colored in one of the exactly $d/4$
possible distinct colors, and the colors are evenly distributed in the plane.
These conclude the proof of Theorem~\ref{Theorem_rmodd}.
\hfill\qed

\medskip

Figure~\ref{FigureColoredOrbits} displays several examples of color alternation effects caused by different orbit lengths resulting from the successive application of transformations $\K_1$ and $\K_2$. 
The patterns produced expand homothetically in $\ZZ^2$,
and the estimates~\eqref{eqcongruentdodd}, \eqref{eqcongruent2edvd},
and~\eqref{eqcongruent4edvd} show that the colors are spread in uniform quantities for each $d\ge 2$.
While the above estimates were made for the same set of orbit generators, 
in the representations in Figure~\ref{FigureColoredOrbits}, 
one can also see the various other shapes produced when the orbit generators come from different configurations.



\end{document}